\documentclass[12pt,a4paper,leqno]{article}
\usepackage{amsmath}
\usepackage{amssymb}
\usepackage{amsthm}
\usepackage{enumerate}
\usepackage{graphicx}

\allowdisplaybreaks
\newcommand{\RR}{\mathbb{R}}
\newcommand{\E}{\mathcal E}
\newcommand{\kal}[1]{\widehat{#1}}
\newcommand{\nul}{\setminus\{0\}}
\newcommand{\uhz}{\mathcal{U}}
\newcommand{\lin}{\mathrm{L}}
\newcommand{\kr}{\mathrm{Ker}}
\newcommand{\vm}{\mathfrak{X}}
\newcommand{\vl}{^\mathsf{v}}
\newcommand{\tl}{^\mathsf{c}}
\newcommand{\hl}{^\mathrm{h}}
\newcommand{\Eh}{\mathcal{H}}
\newcommand{\Ve}{\mathcal{V}}
\newcommand{\hh}{\mathbf{h}}
\newcommand{\splt}{\mathring{T}}
\newcommand{\splp}{\mathring{\pi}}
\newcommand{\lie}{\mathcal{L}}
\newcommand{\hll}[1]{\widetilde{#1}}
\newcommand\pr{\mathop{\mathrm{pr}}\nolimits}
\newcommand\Sec{\mathop{\mathrm{Sec}}\nolimits}
\newcommand\vlift{\mathop{\mathrm{vl}}\nolimits}
\newcommand{\parcv}[2]{\frac{\partial#1}{\partial#2}}

\newcounter{unit}
\newenvironment{unit}[1][\arabic{section}.\arabic{unit}]{\addtocounter{unit}1\begin{trivlist}
\item[\hskip \labelsep {\bfseries #1.}]}{\end{trivlist}}

\newtheorem{fact}{FACT}
\newtheorem{lemma}{Lemma}
\newtheorem{claim}{Claim}
\theoremstyle{definition}
\newtheorem{prop}[lemma]{Proposition}
\newtheorem{cor}[lemma]{Corollary}
\newtheorem*{remark}{Remark}

\theoremstyle{remark}
\newtheorem*{coorddesc}{Coordinate description}
\newtheorem*{note}{Note}

\makeatletter 
\def\fvecfill{$\m@th\hbox{\raisebox{-5.75pt}[1.5pt][0pt]{$\!\mathord\ulcorner$}}
\mkern-5mu
\cleaders\hbox{\raisebox{-0.5pt}[1.5pt][0pt]{$\!\mathord-$}}\hfill
\mkern-5mu
\mathord{\hbox{\raisebox{-5.75pt}[1.5pt][0pt]{$\!\mathord\urcorner\!$}}}$}
\makeatother
\def\fentkapcs { \overset{\hbox{\fvecfill}} } 

\begin{document}

\title{Ten ways to Berwald manifolds -- and some steps beyond\thanks{The first two authors were supported by National Science Research Foundation OTKA No.\ NK 81402.}}
\author{J.~Szilasi\and R.~L.~Lovas\and D.~Cs.~Kert\'esz}
\date{}
\maketitle

\begin{abstract}
After summarizing some necessary preliminaries and tools, including Berwald derivative and Lie derivative in pull-back formalism, we present ten equivalent conditions, each of which characterizes Berwald manifolds among Finsler manifolds. These range from Berwald's classical definition to the existence of a torsion-free covariant derivative on the base manifold compatible with the Finsler function and Aikou's characterization of Berwald manifolds. Finally, we study some implications of V.~Matveev's observation according to which quadratic convexity may be omitted from the definition of a Berwald manifold. These include, among others, a generalization of Z.~I.~Szab\'o's well-known metrization theorem, and leads also to a natural generalization
of Berwald manifolds, to \emph{Berwald\,--\,Matveev manifolds}.
\end{abstract}



\section{Introduction}
Positive definite Berwald manifolds constitute the conceptually simplest and the best understood class of Finsler manifolds. Their conceptual simplicity is due to the fact that Berwald manifolds are affinely connected manifolds at the same time, whose parallelism structure is related to the normed vector space structure of the tangent spaces in the most natural manner: parallel translations are norm preserving. Berwald himself called a Finsler manifold an `affinely connected space' if
\begin{enumerate}
\item[$(B{})$] \emph{the connection parameters arising from the geodesic equation are independent of the direction arguments.}
\end{enumerate}
It turns out that the affine connection of a Berwald manifold is the Levi-Civita connection of a Riemannian metric on the base manifold. This key observation of Z.~I.~Szab\'o is the starting point of his structure theorem on Berwald manifolds \cite{Szabo}.

Our decision to write a comprehensive survey concerning Berwald's condition $(B{})$ was strongly motivated by some e-mails between Vladimir Matveev and the first author. A quotation from a letter of Matveev:

`I always thought that a Finsler manifold is Berwald \emph{if and only if} there exists a torsion-free affine connection whose transport preserves the Finsler function $F$. Is the statement correct? If yes, do you have a reference where it is written?

Of course I understand that a Berwald metric (in a standard definition) does have the above property: indeed, in this case the Chern connection is actually an affine connection, and it preserves the Finsler function $F$. Thus, my question is essentially whether the existence of an affine connection preserving the Finsler function implies that the metric is Berwald\dots'
He also mentioned that the question had appeared in a discussion with Marc Troyanov.

Our answer was affirmative. We did not find, however, any reference where the statement was formulated explicitly and proved in a simple and self-contained way.

To our request Matveev also sketched a possible proof, found by him and Troyanov. Although their argumentation was not elaborated in every detail, we found it nice and original. We thought, however, that it used rather heavy tools from Riemannian geometry to a quite simple problem, and depended too strongly on the assumption of positive definiteness.

Since the question is natural and important, we believed it useful to present a proof which is self-contained as possible, and which uses only the simplest tools of Finsler geometry (and connection theory). Moreover, besides the property formulated by Matveev, we present nine other properties that characterize Berwald manifolds among Finsler manifolds. We do this partly for the sake of completeness, since some other properties are folklore, and some others are easily accessible (see, e.g., ref \cite{SzV}). On the other hand, in several problems, it is advantageous to have an appropriate version of Berwaldian property $(B{})$.

We wrote this paper not only, or not primarily, for Finsler geometers, and we hope that anyone with a basic knowledge of differential geometry can fairly easily read it and will find it indeed useful. So in sections 2--4 we collect the most necessary preparatory material concerning sprays, Ehresmann connections, (nonlinear) parallel translations, and some basic facts on Finsler functions. In section 5 we formulate and prove nine equivalents of $(B{})$; the first of them is just a more precise reformulation of the relevant property. In section 6 we present a detailed, index-free proof of a further nice and important characterization of Berwald manifolds among Finsler manifolds, discovered by T.~Aikou \cite{Aikou}. To do this, we need a concept of `Lie derivative along the tangent bundle projection', which we also explain here.

In section 7 we leave the realm of classical Berwald manifolds. In references \cite{M1,MRTZ,M2} Matveev and his collaborators drew attention to the fact that the metrization of the affine connection of a Berwald manifold discovered by Z.~I.~Szab\'o may be carried out in a more general setting. Namely, the quadratic convexity of a Finsler function assured by our conditions $(F_1)$--$(F_4)$ in section 4, may be weakened substantially. This
observation leads to the less restrictive notion of \emph{Berwald\,--\,Matveev
manifold}.  
We devote the greater part of the concluding section to the averaged metric construction explained first in \cite{MRTZ} (and later also in \cite{M1}), and to a completion of the proof of Theorem 1 in Matveev's paper \cite{M1}. In our arguments we utilize a trick which we learnt from Matveev during an after-lunch conversation in Debrecen, April 2011. Finally, we exhibit a further method to associate a Riemannian metric to a Berwald\,--\,Matveev manifold in a natural way, applying Loewner ellipsoids.

\section{Notations and definitions}
\begin{unit}
Throughout the paper, $M$ will be an $n$-dimensional ($n\geq1$) smooth manifold whose underlying topological space is Hausdorff, second countable and connected. $C^\infty(M)$ is the real algebra of smooth functions on $M$.
\end{unit}

\begin{unit}
The tangent bundle of $M$ will be denoted by $\tau:TM\rightarrow M$. Analogously, $\tau_{TM}:TTM\rightarrow TM$ will stand for the tangent bundle of the tangent manifold $TM$. The shorthand for these vector bundles will be $\tau$ and $\tau_{TM}$, respectively. The vector fields on $M$ form a $C^\infty(M)$-module, which will be denoted by $\mathfrak X(M)$. $o\in\mathfrak X(M)$ is the zero vector field on $M$. The \emph{deleted bundle} for $\tau$ is the fibre bundle $\mathring\tau:\mathring TM\rightarrow M$, where $\mathring TM:=TM\setminus o(M)$, $\mathring\tau:=\tau\upharpoonright\mathring TM$. The Lie bracket $[X,Y]$ of $X,Y\in\mathfrak X(M)$ is the unique vector field on $M$ satisfying
\[[X,Y](f)=X(Yf)-Y(Xf),\quad f\in C^\infty(M).\]
\end{unit}

\begin{unit}
If $\varphi:M\rightarrow N$ is a smooth mapping between smooth manifolds, its derivative is $\varphi_*:TM\rightarrow TN$. Two vector fields $X\in\mathfrak X(M)$ and $Y\in\mathfrak X(N)$ are $\varphi$-related if $\varphi_*\circ X=Y\circ\varphi$; in this case we write $X\raisebox{-2pt}[0pt][-7pt]{$\begin{array}{c}\sim\\[-10pt]\scriptstyle\varphi\end{array}$} Y$. A vector field $\xi$ on $TM$ is \emph{vertical} if $\xi\raisebox{-2pt}[0pt][-7pt]{$\begin{array}{c}\sim\\[-10pt]\scriptstyle\tau\end{array}$}o$. The vertical vector fields form a (finitely generated) module, denoted by $\mathfrak X^{\mathsf v}(TM)$, over the ring $C^\infty(TM)$, which is also a subalgebra of the real Lie algebra $\mathfrak X(TM)$.
\end{unit}

\begin{unit}\label{liouv}
The \emph{vertical lift} of a function $f\in C^\infty(M)$ in $C^\infty(TM)$ is $f^{\mathsf v}:=f\circ\tau$, the \emph{complete lift} of $f$ is $f^{\mathsf c}\in C^\infty(TM)$ defined by
\[f^{\mathsf c}(v):=v(f),\textrm{\emph{ if }} v\in TM.\]
There exists a canonical vertical vector field $C$ on $TM$ such that
\[Cf^{\mathsf c}:=f^{\mathsf c},\textrm{\emph{ for all }} f\in C^\infty(M).\]
$C$ is said to be the \emph{Liouville vector field} (or \emph{radial vector field}) on $TM$.
\end{unit}

\begin{unit}
Let $X$ be a vector field on $M$. The \emph{vertical lift} $X^{\mathsf v}\in\mathfrak X^{\mathsf v}(TM)$ of $X$ is the unique vertical vector field on $TM$ such that
\[X^{\mathsf v}f^{\mathsf c}=(Xf)^{\mathsf v}\textrm{\emph{ for all }} f\in C^\infty(M).\]
The \emph{complete lift} $X^{\mathsf c}\in\mathfrak X(TM)$ of $X$ is the unique vector field on $TM$ such that
\[X^{\mathsf c}f^{\mathsf c}=(Xf)^{\mathsf c},\quad X^{\mathsf c}f^{\mathsf v}=(Xf)^{\mathsf v}\textrm{\emph{ for all }} f\in C^\infty(M).\]
If $(X_i)_{i=1}^n$ is a local frame for $TM$, then $\big(X_i^{\mathsf v},X_i^{\mathsf c}\big)_{i=1}^n$ is a local frame for $TTM$, therefore
\begin{quote}
\emph{in order to define a tensor field on $TM$, it is sufficient to specify its action on vertical and complete lifts of vector fields on $M$.}
\end{quote}
Thus there exists a unique type $(1,1)$ tensor field $\mathbf J$ on $TM$ such that
\[\mathbf J X^{\mathsf v}=0,\quad \mathbf JX^{\mathsf c}=X^{\mathsf v}\textrm{\emph{ for all }} X\in \mathfrak X(M).\]
$\mathbf J$ is said to be the \emph{vertical endomorphism} of $\mathfrak X(TM)$ (or of $TTM$).
\end{unit}

\begin{unit}
$d$ denotes the operator of exterior derivative, defined on a function \\$f\in C^\infty(M)$ and a $1$-form $\omega\in\mathfrak X^*(M)$ by
\[df(X):=Xf\quad\textrm{\emph{and}}\quad d\omega(X,Y):=X\omega(Y)-Y\omega(X)-\omega\big([X,Y]\big)\]
$\big(X,Y\in\mathfrak X(M)\big)$.

The substitution operator $i_X$, associated to a vector field $X\in\mathfrak X(M)$, acts on a type $(0,k)$ or $(1,k)$ ($k\in\mathbb N\setminus\{0\}$) tensor field $A$ on $M$ by
\[i_XA(X_1,\dots,X_{k-1}):=A(X,X_1,\dots,X_{k-1}).\]
\end{unit}

\begin{unit} To any type $(1,1)$ tensor field
\[A\in\mathcal T_1^1(M)\cong \mathop{\mathrm{End}}\big(\mathfrak X(M)\big)\]
we associate a vertical vector field $\overline A\in\mathfrak X^{\mathsf v}(TM)$, by prescribing its action on the complete lifts of smooth functions on $M$ by
\[\big(\overline Af^{\mathsf c}\big)(v):=A_{\tau(v)}(v)f;\quad f\in C^\infty(M),\ v\in TM.\]
Then we have
\[\big[C,\overline A\,\big]=0.\]
\end{unit}

\begin{unit}
By a \emph{semispray} for $M$ we mean a mapping
\[S:TM\longrightarrow TTM\]
satisfying the following conditions:
\begin{enumerate}[($S_1$)]
\item $\tau_{TM}\circ S=1_{TM}$.
\item $S$ is of class $C^1$ on $TM$, smooth on $\mathring TM$.
\item $\mathbf JS=C$.
\end{enumerate}
A semispray is called a \emph{spray} if it also satisfies
\begin{enumerate}
\item [($S_4$)] $[C,S]=S$.
\end{enumerate}
If a spray is of class $C^2$ (and hence smooth) on $TM$, we speak of an \emph{affine spray}.
\end{unit}

\begin{unit}
If $(\mathcal U,u)=\big(\mathcal U,(u^i)_{i=1}^n\big)$ is a chart on $M$, then
\[\big(\tau^{-1}(\mathcal U),(x^i,y^i)_{i=1}^n\big),\quad x^i:=(u^i)^{\mathsf v},\ y^i:=(u^i)^{\mathsf c}\]
is a chart on $TM$, called the \emph{chart induced by} $(\mathcal U,u)$. In our (not too frequent) coordinate calculations \emph{Einstein's summation convention} will be applied: any index occurring twice, once up, once down, is summed over.
\end{unit}

To conclude this section, we present the coordinate expressions of some objects introduced above.
\begin{enumerate}[(i)]
\item If $\xi\in\mathfrak X^{\mathsf v}(TM)$, then
\[\xi\upharpoonright\tau^{-1}(\mathcal U)=\xi^{n+i}\frac{\partial}{\partial y^i},\quad \xi^{n+i}\in C^\infty\big(\tau^{-1}(\mathcal U)\big).\]
\item In the induced coordinates, the Liouville vector field takes the form
    \[C\upharpoonright\tau^{-1}(\mathcal U)=y^i\frac{\partial}{\partial y^i}.\]
\item If $f\in C^\infty(\mathcal U)$, its complete lift is
\[f^{\mathsf c}=y^i\left(\frac{\partial f}{\partial u^i}\circ\tau\right)=(u^i)^{\mathsf c}\left(\frac{\partial f}{\partial u^i}\right)^{\mathsf v}.\]
\item If $X\in\mathfrak X(M)$, $X\upharpoonright\mathcal U =X^i\frac{\partial }{\partial u^i}$, then
\begin{align*}
X^{\mathsf v}\upharpoonright\tau^{-1}(\mathcal U)&=\big(X^i\circ\tau\big)\frac{\partial}{\partial y^i},\\
X^{\mathsf c}\upharpoonright\tau^{-1}(\mathcal U)&=\big(X^i\circ\tau\big)\frac{\partial}{\partial x^i}+y^j\left(\frac{\partial X^i}{\partial u^j}\circ\tau\right)\frac{\partial}{\partial y^i}.
\end{align*}
\item If $A\in\mathcal T_1^1(M)$, $A\left(\frac{\partial}{\partial u^j}\right)=A_j^i\frac{\partial}{\partial u^i}$ ($j\in\{1,\dots,n\}$), then
\[\overline A\upharpoonright\tau^{-1}(\mathcal U)=y^j\left(A^i_j\circ\tau\right)\frac{\partial}{\partial y^i}.\]
\item If $S:TM\rightarrow TTM$ is a semispray, then \[S\upharpoonright\tau^{-1}(\mathcal U)=y^i\frac{\partial}{\partial x^i}-2G^i\frac{\partial}{\partial y^i},\]
    where the functions $G^i:\tau^{-1}(\mathcal U)\rightarrow\mathbb R$ are of class $C^1$, and smooth on $\mathring\tau^{-1}(\mathcal U)$.
\end{enumerate}
\section{Ehresmann connections and parallel translations}
\setcounter{unit}0
\begin{unit}
Consider the vector bundle
\[
\pi:TM\times_M TM\to TM,
\]
where $TM\times_M TM:=\{(u,v)\in TM\times TM\mid\tau(u)=\tau(v)\}$, and $\pi$ is the restriction of the canonical projection $\pr_1:TM\times TM\to TM$, $(u,v)\mapsto u$ onto $TM\times_M TM$.

In terms of the theory of bundles, $\pi$ is just the pull-back of the tangent bundle $\tau:TM\rightarrow M$ by $\tau$. The $C^\infty(TM)$-module of sections of $\pi$ will be denoted by $\Sec(\pi)$. For any vector field $X$ on $M$, the mapping
\[\kal X:v\in TM\longmapsto \kal X(v):=\big(v,X(\tau(v))\big)\in TM\times_MTM\]
is a section of $\pi$, called a \emph{basic section}.
Basic sections generate the module $\Sec(\pi)$ in the sense that locally any section in $\Sec(\pi)$ can be obtained as a $C^\infty(TM)$-linear combination of basic sections. In particular, if $\left(\mathcal U,\left(u^i\right)_{i=1}^n\right)$ is a chart of $M$, then $\left(\widehat{\frac\partial{\partial u^i}}\right)_{i=1}^n$ is a frame of $TM\times_M TM$ over $\tau{}^{-1}(\mathcal U)$.

We have a canonical $C^\infty(TM)$-linear isomorphism
\[
\vlift:\Sec(\pi)\to\mathfrak X^{\mathsf v}(TM),
\]
called the \emph{vertical lift}, given on the basic sections by
\[
\vlift\big(\widehat X\big):=X^{\mathsf v},\ X\in\mathfrak X(M).
\]

We shall also need the pull-back of the tangent bundle $\tau: TM\rightarrow M$ by the mapping $\mathring\tau:\splt M\rightarrow M$; this is the vector bundle
\[\splp:\splt M\times_MTM\rightarrow\splt M.\]
We write $\Sec(\splp)$ for the $C^\infty(\splt M)$-module of its sections. As before, any vector field on $M$ determines a basic section in $\Sec(\splp)$.

It will be useful to extend the derivative $\tau_*:TTM\rightarrow TM$ of $\tau$ into a mapping
\[\hll\tau_*:TTM\longrightarrow TM\times_MTM\]
given by
\[\hll\tau_*(w):=\big(v,(\tau_*)_v(w)\big),\textrm{ \emph{if} } w\in T_vTM.\]
\end{unit}

\begin{unit}
By an \emph{Ehresmann connection in $TM$} we mean a mapping
\[
\mathcal H:TM\times_M TM\to TTM
\]
satisfying the following conditions:
\begin{enumerate}[(C1)]
\item
$\mathcal H$ is fibre-preserving and fibrewise linear, i.e., for each $v\in TM$, $w,w_1,w_2\in T_{\tau(v)}M$, $\lambda_1,\lambda_2\in\mathbb R$,
\[
\mathcal H(v,w)\in T_vTM,
\]
and
\[
\mathcal H(v,\lambda_1w_1+\lambda_2w_2)
=\lambda_1\mathcal H(v,w_1)+\lambda_2\mathcal H(v,w_2).
\]
\item
$\hll\tau_*\circ\Eh=1_{TM\times_MTM}$, or equivalently, $(\tau_*)_v\big(\Eh(v,w)\big)=w$, \emph{for all}\\ $v\in TM$, $w\in T_vTM$.
\item
$\mathcal H$ is smooth over $\mathring TM\times_M TM$.
\item
For each $p\in M$ and $v\in T_pM$, $\mathcal H(o(p),v)=(o_*)_p(v).$
\end{enumerate}

If
\[\Eh_v:=\Eh\upharpoonright \{v\}\times T_{\tau(v)}M\quad (v\in\splt M),\]
then by (C1) and (C2) $\Eh_v$ is an injective linear mapping of \\ $T_{\tau(v)}M\cong\{v\}\times T_{\tau(v)}M$ into $T_v\splt M$, therefore
\[H_vTM:=\mathop{\textrm{Im}}(\Eh_v)\]
is an $n$-dimensional subspace of $T_v\splt M$, called the \emph{horizontal subspace} of $T_v\splt M$ with respect to $\Eh$. If
\[V_v\splt M:=\mathop{\textrm{Ker}}(\tau_*)_v\]
is the (canonical) vertical subspace of $T_v\splt M$, then we have the direct decomposition
\[T_v\splt M=V_v\splt M\oplus H_v\splt M.\]

The mapping
\[\hh:=\Eh\circ\hll\tau_*:TTM\longrightarrow TTM\]
is a projection operator: fibrewise linear and $\hh^2=\hh$. $\hh$ is called the \emph{horizontal projection} associated to $\Eh$.

The \emph{horizontal lift} of a vector field $X$ on $M$ with respect to $\mathcal H$ (or the \emph{\mbox{$\mathcal H$-horizontal} lift}, briefly the horizontal lift of $X$) is the vector field \\
$X^{\mathrm h}\in\mathfrak X\big(\mathring TM\big)$ defined by
\[
X^{\mathrm h}(v):=\mathcal H(v,X(\tau(v))),\ v\in\mathring TM.
\]

If $(X_i)_{i=1}^n$ is a local frame of $TM$, then $(X_i\vl,X_i\hl)_{i=1}^n$ is a local frame of $T\splt M$ (cf. 2.5). We define a $C^\infty(\splt M)$-linear mapping
\[\Ve:\vm(\splt M)\longrightarrow\Sec(\splp),\]
called the \emph{vertical mapping} associated to $\Eh$, specifying its action on the vertical and horizontal lifts of vector fields on $M$ by
\[\Ve(X\vl)=\kal X,\quad \Ve(X\hl)=0;\quad X\in\vm(M).\]

An Ehresmann connection $\mathcal H$ is said to be \emph{homogeneous} if $\left[X^{\mathrm h},C\right]=0$ for all $X\in\mathfrak X(M)$; \emph{torsion-free} if $\left[X^{\mathrm h},Y^{\mathsf v}\right]-\left[Y^{\mathrm h},X^{\mathsf v}\right]-[X,Y]^{\mathsf v}=0$ for all $X,Y\in\mathfrak X(M)$.

By a \emph{linear connection} on $TM$ (or, by an abuse of language, on $M$) we mean a homogeneous Ehresmann connection which is of class $C^1$ (and hence smooth) over $TM\times_M TM$. The motivation of this terminology will be clear from the coordinate description below.

\begin{remark}
We draw a definite distinction between an Ehresmann connection and a (possibly nonlinear) covariant derivative operator in Koszul's sense, although they are two sides of the same coin.
\end{remark}

\begin{coorddesc}
Let $\mathcal H$ be an Ehresmann connection in $TM$. Specify a chart $\left(\mathcal U,\left(u^i\right)_{i=1}^n\right)$ of $M$, and consider the induced chart $\left(\tau^{-1}(\mathcal U),\left(x^i,y^i\right)_{i=1}^n\right)$ of $TM$. $\mathcal H$ determines unique functions
\[
N^i_j:\tau^{-1}(\mathcal U)\to\mathbb R,\ i,j\in\{1,\dots,n\},
\]
smooth on $\tau^{-1}(\mathcal U)\cap\mathring TM$, such that for each $j\in\{1,\dots,n\}$,
\[
\bigg(\frac{\partial}{\partial u^j}\bigg)\hl(v)=
\left(\frac\partial{\partial x^j}\right)_v
-N^i_j(v)\left(\frac\partial{\partial y^i}\right)_v.
\]
They are called the \emph{Christoffel symbols} or the \emph{connection parameters} for $\mathcal H$ with respect to the given charts. If $\mathcal H$ is homogeneous, the $N^i_j$'s are positive-homogeneous of degree 1; if $\mathcal H$ is linear, they can be written in the form
\[
N^i_j=\left(\Gamma^i_{jk}\circ\tau\right)y^k,
\ \Gamma^i_{jk}\in C^\infty(\mathcal U),
\]
i.e., they are linear functions on each tangent space.

For the vertical mapping associated to $\Eh$ we obtain
\begin{align*}
\Ve\bigg(\parcv{}{x^j}\bigg)&=\Ve\bigg(\bigg(\parcv{}{u^j}\bigg)\hl+N^i_j\bigg(\parcv{}{u^i}\bigg)\vl\bigg)=N^i_j\kal{\parcv{}{u^i}},\\
\Ve\bigg(\parcv{}{y^j}\bigg)&=\kal{\parcv{}{u^j}}\quad \big(j\in\{1,\dots,n\}\big).
\end{align*}
\end{coorddesc}
\end{unit}

\begin{unit}
We recall two basic examples of constructing an Ehresmann connection.
\begin{enumerate}[(a)]
\item
\emph{Crampin's construction.} Any semispray $S$ for $M$ induces a torsion-free Ehresmann connection such that
\[
X^{\mathrm h}=\frac 12(X^{\mathsf c}+[X^{\mathsf v},S])\mbox{ \emph{for all} }X\in\mathfrak X(M).
\]
\item
\emph{Ehresmann connection from a covariant derivative.} Let $D$ be a covariant derivative operator on $M$. There exists a unique Ehresmann connection $\mathcal H_D$ in $TM$ such that the horizontal lift $X^{\mathrm h_D}$ of a vector field on $M$ with respect to $\mathcal H_D$ is given by
\[
X^{\mathrm h_D}=X^{\mathsf c}-\overline{DX},
\]
where $\overline{DX}\in\mathfrak X^{\mathsf v}(TM)$ is the vertical vector field constructed from the covariant differential $DX$, as described in 2.7. Since $[X^{\mathsf c},C]=0$ and $\left[\overline{DX},C\right]=0$, it follows that $\mathcal H_D$ is a homogeneous Ehresmann connection. If $\left(\mathcal U,\left(u^i\right)_{i=1}^n\right)$ is a chart on $M$, and the Christoffel symbols of $D$ on $\mathcal U$ are $\Gamma^i_{jk}\in C^\infty(\mathcal U)$, then
\[
\left(\frac\partial{\partial u^k}\right)^{\mathrm h_D}
=\frac\partial{\partial x^k}
-\left(\Gamma^i_{jk}\circ\tau\right)y^j\frac\partial{\partial y^i},
\]
so $\mathcal H_D$ is a linear connection. An easy calculation shows that
\[
\left[X^{\mathsf v},\overline{DY}\right]=(D_XY)^{\mathsf v}\mbox{ \emph{for all} }X,Y\in\mathfrak X(M).
\]
Thus
\begin{align*}
&\left[X^{\mathrm h_D},Y^{\mathsf v}\right]-\left[Y^{\mathrm h_D},X^{\mathsf v}\right]-[X,Y]^{\mathsf v}
=[X^{\mathsf c},Y^{\mathsf v}]-[Y^{\mathsf c},X^{\mathsf v}]-[X,Y]^{\mathsf v}\\
&\quad-\left[\overline{DX},Y^{\mathsf v}\right]+\left[\overline{DY},X^{\mathsf v}\right]
=[X,Y]^{\mathsf v}-[Y,X]^{\mathsf v}-[X,Y]^{\mathsf v}\\
&\quad+(D_YX)^{\mathsf v}-(D_XY)^{\mathsf v}=-(D_XY-D_YX-[X,Y])^{\mathsf v},
\end{align*}
therefore $\mathcal H_D$ is torsion-free if, and only if, $D$ is torsion-free.
\end{enumerate}
\end{unit}

\begin{unit}
Let a \emph{homogeneous} Ehresmann connection $\mathcal H$ be specified in $TM$. Let $I$ be an open interval containing 0. Consider a (smooth) curve $\gamma:I\to M$ and a vector field $X:I\to TM$ along $\gamma$ (then $\tau\circ X=\gamma$). $X$ is said to be \emph{parallel} along $\gamma$ with respect to $\mathcal H$ (\emph{$\mathcal H$-parallel}, or simply parallel) if
\[
\dot X(t)=\mathcal H(X(t),\dot\gamma(t))\mbox{ \emph{for all} }t\in I,
\]
briefly, if $\dot X=\mathcal H(X,\dot\gamma)$.

We note that if $\gamma:I\rightarrow M$ is an integral curve of a vector field $Z\in\vm(M)$, and $X:I\rightarrow TM$ is  $\Eh$-parallel along $\gamma$, then $X$ is an integral curve of $Z\hl$, i.e., $\dot X=Z\hl\circ X$.

Indeed, at any point $t\in I$, we have on the one hand
\[\dot X(t)=\Eh\big(X(t),\dot\gamma(t)\big)=\Eh\big(X(t),Z(\gamma(t))\big);\]
on the other hand
\[Z\hl\big(X(t)\big):=\Eh\big(X(t),Z(\tau(X(t)))\big)=\Eh\big(X(t),Z(\gamma(t))\big).\]

The general existence and uniqueness theorem for solutions of homogeneous ODEs (see e.g.\ \cite{Barthel}) guarantees that, for any tangent vector $v\in\mathring T_{\gamma(0)}M$, there exists a unique parallel vector field $X$ along $\gamma$ such that $X(0)=v$. If $t\in I$, the mapping
\[
(P_\gamma)_0^t:T_{\gamma(0)}M\to T_{\gamma(t)}M,
\ v\mapsto(P_\gamma)_0^t(v):=X(t)
\]
is called \emph{parallel translation along $\gamma$} from $p=\gamma(0)$ to $q=\gamma(t)$. $(P_\gamma)_0^t$ is a positive-homogeneous diffeomorphism from $\mathring T_pM$ to $\mathring T_qM$.

By an ($\Eh$-) \emph{horizontal lift of a (smooth)} curve $\gamma:I\rightarrow M$ we mean a curve
\[\gamma\hl:I\longrightarrow \splt M\]
such that
\[\tau\circ\gamma\hl=\gamma \textrm{\emph{ and }}\dot\gamma\hl(t)\in H_{\gamma\hl(t)}\splt M\textrm{\emph{ for each }}t\in I.\]
\begin{lemma}\label{lemparal}
If $\gamma\hl$ is an $\Eh$-horizontal lift of a curve $\gamma:I\rightarrow M$, then $\gamma\hl$ is parallel along $\gamma$, i.e.,
\[\dot\gamma\hl(t)=\Eh\big(\gamma\hl(t),\dot\gamma(t)\big),\quad t\in I.\]
\end{lemma}
\begin{proof}
Since
\begin{align*}
\dot\gamma(t)&=(\gamma_*)_t\bigg(\frac{d}{dr}\bigg)_t=\big((\tau\circ\gamma\hl)_*\big)_t\bigg(\frac{d}{dr}\bigg)_t=(\tau_*)_{\gamma\hl(t)}\big((\gamma\hl)_*\big)_t\bigg(\frac{d}{dr}\bigg)_t=\\
&=(\tau_*)_{\gamma\hl(t)}\dot\gamma\hl(t),
\end{align*}
we obtain
\[\dot\gamma\hl(t)=\hh\big(\dot\gamma\hl(t)\big)=\Eh\big(\hll\tau_*(\dot\gamma\hl(t))\big)=\Eh\big(\gamma\hl(t),(\tau_*)_{\gamma\hl(t)}\dot\gamma\hl(t)\big)=\Eh\big(\gamma\hl(t),\dot\gamma(t)\big).\]
\end{proof}
\begin{lemma}\label{lemflow}
Let $X$ be a vector field on $M$, and let
\[\varphi:W\subset\mathbb R\times M\longrightarrow M,\quad (t,p)\longmapsto \varphi(t,p)\]
be the flow generated by $X$. If\/ $\Eh$ is an Ehresmann connection in $TM$, then the flow generated by the $\Eh$-horizontal lift $X\hl$ of $X$ is the mapping
\[\varphi\hl:\hll W\longrightarrow \splt M,\ (t,v)\longmapsto \varphi\hl(t,v):=\varphi_{v}\hl(t),\]
where
\[\hll W:=\big\{(t,v)\in\mathbb R\times\splt M\big|(t,\tau(v))\in W\big\},\]
and for any fixed $v\in\splt M$, $\varphi_{v}\hl$ is the horizontal lift of the curve $\varphi_{\tau(v)}$ defined by $\varphi_{\tau(v)}(t):=\varphi(t,\tau(v))$, starting from $v$.
\end{lemma}
\begin{proof}
We have to check that for any fixed $v\in\splt M$,
\[\dot{\fentkapcs{\varphi_v\hl}}=X\hl\circ \varphi_v\hl.\]
If $t$ is in the domain of $\varphi_v\hl$, then
\begin{align*}
X\hl\big(\varphi_v\hl(t)\big):&=\Eh\big(\varphi_v\hl(t),X\circ\tau(\varphi_v\hl(t))\big)=\\
&=\Eh\big(\varphi_v\hl(t),X\circ\varphi_{\tau(v)}(t)\big)=\\
&=\Eh\big(\varphi_v\hl(t),\dot\varphi_{\tau(v)}(t)\big)\mathop{=}^{\textrm{Lemma \ref{lemparal}}}\\
&=\dot{\fentkapcs{\varphi_v\hl}}(t).
\end{align*}

\end{proof}

As in the classical theory of linear connections, a regular curve $\gamma:I\to M$ is said to be a \emph{geodesic} of an Ehresmann connection $\mathcal H$ if its velocity field $\dot\gamma$ is $\mathcal H$-parallel, i.e.,
\[
\ddot\gamma=\mathcal H(\dot\gamma,\dot\gamma).
\]

In local coordinates, the equations of parallel vector fields become
\[
\tag{P}X^i{}'+\left(N^i_j\circ X\right)\gamma^j{}'=0\quad(i\in\{1,\dots,n\}),
\]
if $X\upharpoonright\gamma^{-1}(\mathcal U)=X^i\left(\frac\partial{\partial u^i}\circ\gamma\right)$, $\gamma^i:=u^i\circ\gamma$, and, as above, the functions $N^i_j$ are the Christoffel symbols for $\mathcal H$. In particular, the \emph{geodesic equations} take the form
\[
\tag{G}\gamma^i{}''+\left(N^i_j\circ\dot\gamma\right)\gamma^j{}'=0,
\quad i\in\{1,\dots,n\}.
\]

We need the following simple observation.
\begin{lemma}
If $D$ is a torsion-free covariant derivative operator on $M$, and $\mathcal H_D$ is the Ehresmann connection induced by $D$, then a vector field $X:I\to TM$ along a curve $\gamma:I\to M$ is parallel with respect to $D$ if, and only if, it is $\mathcal H_D$-parallel.
\end{lemma}

Indeed, if the Christoffel symbols for $D$ are the functions $\Gamma^i_{jk}\in C^\infty(\mathcal U)$, then the Christoffel symbols for $\mathcal H_D$ are $N^i_j=\left(\Gamma^i_{jk}\circ\tau\right)y^k$, hence
\[
N^i_j\circ X=\left(\left(\Gamma^i_{jk}\circ\tau\right)y^k\right)\circ X
=\left(\Gamma^i_{jk}\circ\tau\circ X\right)(y^k\circ X)
=\left(\Gamma^i_{jk}\circ\gamma\right)X^k,
\]
so equations (P) become
\[
X^i{}'+\left(\Gamma^i_{jk}\circ\gamma\right)\gamma^j{}'X^k=0,
\ i\in\{1,\dots,n\},
\]
which are the familiar equations of parallelism with respect to a covariant derivative operator.
\end{unit}

\begin{unit}
We say that an Ehresmann connection $\mathcal H$ in $TM$ is \emph{compatible} with a $C^1$-function $F:TM\to\mathbb R$ if $dF\circ\mathcal H=0$, or, equivalently, if
\[
X^{\mathrm h}F=0\mbox{ \emph{for all} }X\in\mathfrak X(M).
\]

\begin{lemma}
Assume that\/ $\mathcal H$ is a homogeneous Ehresmann connection in $TM$, and let $F:TM\to\mathbb R$ be a $C^1$-function. The following are equivalent:
\begin{enumerate}[(i)]
\item
$\mathcal H$ is compatible with $F$.
\item
For any $\mathcal H$-parallel vector field $X:I\to TM$ along a curve $\gamma:I\to M$, the function
\[
F\circ X:I\to\mathbb R
\]
is constant.
\end{enumerate}
\end{lemma}
\begin{proof}
Let $\frac d{dr}$ be the canonical vector field on $I$ ($r:=1_{\mathbb R}$). $F\circ X$ can be considered as a curve in $\mathbb R$; then for each $t\in\mathbb R$ we have
\begin{align*}
(F\circ X)'(t)&=\dot{\fentkapcs{F\circ X}}(t)
=((F\circ X)_*)_t\left(\frac d{dr}\right)_t\\
&=(F_*)_{X(t)}\circ(X_*)_t\left(\frac d{dr}\right)_t
=(F_*)_{X(t)}\left(\dot X(t)\right)\\
&=(dF)_{X(t)}(\mathcal H(X(t)),\dot\gamma(t)))
=dF\circ\mathcal H(X(t),\dot\gamma(t))
\end{align*}
identifying in our calculation the derivative $(F_*)_{X(t)}$ with the differential $(dF)_{X(t)}$, and taking into account the condition that $X$ is parallel along $\gamma$. The relation so obtained implies immediately that \emph{$F\circ X$ is constant if, and only if, $dF\circ\mathcal H=0$.}
\end{proof}
\end{unit}

\section{Basic facts on Finsler functions}
\setcounter{unit}0

\begin{unit}
A function $F:TM\to\mathbb R$ is said to be a \emph{Finsler function} if
\begin{enumerate}[(F1)]
\item
$F$ is smooth on $\mathring TM$;
\item
$F(\lambda v)=\lambda F(v)$ for all $v\in TM$ and positive real number $\lambda$;
\item
$F(v)>0$ if $v\in\mathring TM$;
\item
the \emph{metric tensor} $g$ defined on basic sections by
\[
g\big(\widehat X,\widehat Y\big):=\frac 12X^{\mathsf v}\big(Y^{\mathsf v}F^2\big)
\quad(X,Y\in\mathfrak X(M))
\]
is fibrewise nondegenerate.
\end{enumerate}

A \emph{Finsler manifold} is a manifold endowed with a Finsler function on its tangent manifold. More formally, a Finsler manifold is a pair $(M,F)$ consisting of a manifold $M$ and a Finsler function $F$ on $TM$. For each $v\in TM$, $F(v)$ is called the \emph{Finsler norm} of $v$.

By conditions (F1) and (F2), $F$ is continuous on $TM$ and vanishes on $o(M)$. It may be shown that our requirements on a Finsler function also imply that the metric tensor is (fibrewise) positive definite \cite{Ker,L}.

\begin{fact}
Let $(M,F)$ be a Finsler manifold. There exists a unique spray $S$ for $M$, called the \emph{canonical spray} of $(M,F)$, defined to be zero on $o(M)$ and satisfying
\[
i_Sd(dF^2\circ \mathbf{J})=-dF^2
\]
on $\mathring TM$.
\end{fact}

\begin{fact}
The torsion-free Ehresmann connection associated to the canonical spray of a Finsler manifold by Crampin's construction \textup{(3.2(a))} is \emph{homogeneous} and \emph{compatible with the Finsler function}.
\end{fact}

For a quite recent index-free proof of this fact we refer to \cite{LSz-Glob}.

\begin{fact}[the uniqueness of the canonical connection]
If a torsion-free, homogeneous Ehresmann connection is compatible with a Finsler function, then it is the canonical connection of the Finsler manifold.
\end{fact}

For a simple recent proof, based on an idea of Z.~I.~Szab\'o, we refer to \cite{SzZ-ArXiV}.

\begin{note}
If $\mathcal H$ is the canonical connection of a Finsler manifold, then, by Lemma~2, \emph{the Finsler norm of a vector remains invariant under $\mathcal H$-parallel translations}.
\end{note}
\end{unit}

\begin{unit}
Let $(M,F)$ be a Finsler manifold with canonical spray $S$ and canonical connection $\mathcal H$.
\begin{enumerate}[(a)]
\item
By a \emph{geodesic} of $(M,F)$ we mean a geodesic of its canonical connection, or, equivalently, a regular curve $\gamma:I\to M$ whose velocity field is an integral curve of the canonical spray:
\[
\ddot\gamma=S\circ\dot\gamma.
\]
\item
$\mathcal H$ induces a covariant derivative operator
\[
\nabla:\mathfrak X\big(\mathring TM\big)
\times\Sec\big(\mathring\pi\big)\to\Sec\big(\mathring\pi\big),
\]
called the \emph{Berwald derivative} of $(M,F)$, such that for all $X,Y\in\mathfrak X(M)$,
\[
\vlift\nabla_{X^{\mathrm h}}\widehat Y=\left[X^{\mathrm h},Y^{\mathsf v}\right],\ \nabla_{X^{\mathsf v}}\widehat Y=0.
\]
Here, the first relation can be written in the equivalent form
\[\nabla_{X\hl}\kal Y=\Ve\big[X\hl,Y\vl\big].\]
\item
The \emph{Berwald curvature} of $(M,F)$ is the type (1,3) tensor $\mathbf B$ on the $C^\infty\big(\mathring TM\big)$-module $\Sec(\mathring\pi)$ such that
\[
\vlift\mathbf B\big(\widehat X,\widehat Y,\widehat Z\big)
=\left[X^{\mathsf v},\left[Y^{\mathrm h},Z^{\mathsf v}\right]\right]
\]
for all $X,Y,Z\in\mathfrak X(M)$.
\end{enumerate}
\end{unit}

\begin{lemma}\label{lemBnab}
If\/ $\nabla$ is the Berwald derivative induced by a torsion-free Ehresmann connection $\Eh$, then for any vector fields $X,Y$ on $M$,
\[\nabla_{X\hl}\kal Y-\nabla_{Y\hl}\kal X=\kal{[X,Y]}.\]
\end{lemma}
\begin{proof}
Applying the definition of $\nabla$ and the torsion-freeness of $\Eh$,
\[\vlift\big(\nabla_{X\hl}\kal Y-\nabla_{Y\hl}\kal X\big)=[X\hl,Y\vl]-[Y\hl,X\vl]=[X,Y]\vl=\vlift\kal{[X,Y]},\]
whence our claim.
\end{proof}
\begin{unit}
For the sake of readers who prefer the language of classical tensor calculus, we present here the coordinate expressions of some important objects introduced above.

Let $(M,F)$ be a Finsler manifold. Choose a chart $\left(\mathcal U,\left(u^i\right)_{i=1}^n\right)$ on $M$, and consider the induced chart $\left(\tau^{-1}(\mathcal U),\left(x^i,y^i\right)_{i=1}^n\right)$ on $TM$.

\begin{enumerate}[(i)]
\item
The components of the metric tensor of $(M,F)$ are the functions
\[
g_{ij}:=g\bigg(\widehat{\frac\partial{\partial u^i}},
\widehat{\frac\partial{\partial u^j}}\bigg)
=\frac 12\frac{\partial^2F^2}{\partial y^i\partial y^j},
\ i,j\in\{1,\dots,n\}.
\]
\item
Over $\tau^{-1}(\mathcal U)$, the canonical spray of $(M,F)$ can be represented in the form
\[
S=y^i\frac\partial{\partial x^i}-2G^i\frac\partial{\partial y^i},
\]
where the \emph{spray coefficients} are
\[
G^i=\frac 14g^{ij}\left(\frac{\partial^2F^2}{\partial x^r\partial y^j}y^r
-\frac{\partial F^2}{\partial x^j}\right);
\ \left(g^{ij}\right):=(g_{ij})^{-1}.
\]
\item
The Christoffel symbols of the canonical connection and the Berwald derivative are
\[
G^i_j:=\frac{\partial G^i}{\partial y^j}\quad
\mbox{and}\quad G^i_{jk}:=\frac{\partial G^i}{\partial y^j\partial y^k},
\]
respectively. Then
\[
\left(\frac\partial{\partial u^j}\right)^{\mathrm h}
=\frac\partial{\partial x^j}-G^i_j\frac\partial{\partial y^i},\quad
\nabla_{\left(\frac\partial{\partial u^j}\right)^{\mathrm h}}\
\widehat{\frac\partial{\partial u^k}}
=G^i_{jk}\widehat{\frac\partial{\partial u^i}}.
\]
The components of the Berwald curvature are given by
\[
\mathbf B\bigg(\widehat{\frac\partial{\partial u^j}},
\widehat{\frac\partial{\partial u^k}},
\widehat{\frac\partial{\partial u^l}}\bigg)
=G^i_{jkl}\widehat{\frac\partial{\partial u^i}},
\quad G^i_{jkl}:=\frac{\partial G^i_{jk}}{\partial y^l},
\]
from which it is clear that $\mathbf B$ is totally symmetric.
\end{enumerate}
\end{unit}

\section{Berwald manifolds}

\renewcommand\labelenumi{\upshape{(B\theenumi)}}
\begin{prop}
\emph{Let $(M,F)$ be a Finsler manifold. The following conditions are equivalent:
\begin{enumerate}
\item
$(M,F)$ is an `affinely connected space' in Berwald's sense \cite{Ber}, that is, the Christoffel symbols $G^i_{jk}$ of the Berwald derivative `depend only on the position'.
\item
The Berwald curvature of $(M,F)$ vanishes.
\item
There exists a covariant derivative operator $D$ on $M$ such that for all $X,Y\in\mathfrak X(M)$,
\[
(D_XY)^{\mathsf v}=\left[X^{\mathrm h},Y^{\mathsf v}\right]
\]
\textup(the horizontal lift is taken with respect to the canonical connection of $(M,F)$\textup). This covariant derivative is torsion-free.
\item
The Berwald derivative $\nabla$ of $(M,F)$ is \emph{$h$-basic} in the sense that there exists a covariant derivative operator $D$ on $M$ such that
\[\vlift \nabla_{X\hl}\kal Y=(D_XY)\vl \quad\textrm{\emph{for all}}\quad X,Y\in\vm(M).\]
\item
The Lie bracket $\left[X^{\mathrm h},Y^{\mathsf v}\right]$ is a vertical lift for any vector fields $X,Y$ on~$M$.
\item
The canonical spray of $(M,F)$ is an affine spray.
\item
The canonical connection of $(M,F)$ is a linear connection.
\item
There exists a torsion-free covariant derivative operator $D$ on $M$ such that the parallel translations with respect to $D$ preserve the Finsler norms of tangent vectors to $M$.
\item
There exists a torsion-free covariant derivative operator $D$ on $M$ such that the geodesics of $D$ coincide with the geodesics of $(M,F)$ as parametrized curves.
\end{enumerate}}
\end{prop}
\begin{proof}
We organize our reasoning according to the following scheme:
\vspace{10pt}
\begin{center}\begin{tabular}{c@{ }c@{ }c@{ }c@{ }c@{ }c@{ }c}
(B1)&$\Longleftrightarrow$&(B2)&$\Longleftrightarrow$
&(B5)&$\Longleftarrow$&(B7)\\
&&&\rotatebox[origin=c]{40}{$\Longleftarrow$}
&\rotatebox[origin=c]{90}{$\Longrightarrow$}&&\rotatebox[origin=c]{90}{$\Longrightarrow$}\\
(B8)&$\Longleftrightarrow$&(B3)&$\Longrightarrow$
&(B4)&&(B6)\\
&&&\rotatebox[origin=c]{-40}{$\Longrightarrow$}
&&\rotatebox[origin=c]{40}{$\Longrightarrow$}&\\[-5pt]
&&&&(B9)&&
\end{tabular}.\end{center}
\vspace{10pt}

(B1)$\iff$(B2)\quad This is obvious, since, as we have just seen, the components of the Berwald curvature are the functions $\frac{\partial G^i_{jk}}{\partial y^l}$.
\vspace{10pt}

(B2)$\Longrightarrow$(B5)\quad If $\mathbf B=0$, then, by 4.2.(c) and 4.3.(iii), for any vector fields $X,Y,Z$ on $M$, we have
\[
\left[\left[X^{\mathrm h},Y^{\mathsf v}\right],Z^{\mathsf v}\right]=0.
\]
Since $X^{\mathrm h}
\raisebox{-2pt}[0pt][-7pt]{$\begin{array}{c}\sim\\[-10pt]\scriptstyle\tau\end{array}$} X$, $Y^{\mathsf v}\raisebox{-2pt}[0pt][-7pt]{$\begin{array}{c}\sim\\[-10pt]\scriptstyle\tau\end{array}$} 0$, $\left[X^{\mathrm h},Y^{\mathsf v}\right]$ is vertical. This vertical vector field commutes with any vertically lifted vector field, which implies easily that $\left[X^{\mathrm h},Y^{\mathsf v}\right]$ is itself a vertical lift.
\vspace{10pt}

(B5)$\Longrightarrow$(B2)\quad If $\left[X^{\mathrm h},Y^{\mathsf v}\right]$ is a vertical lift for each $X,Y\in\mathfrak X(M)$, then for any vector field $Z$ on $M$,
\[
0=\left[Z^{\mathsf v},\left[X^{\mathrm h},Y^{\mathsf v}\right]\right]
=\vlift\mathbf B\big(\widehat Z,\widehat X,\widehat Y\big)
=\vlift\mathbf B\big(\widehat X,\widehat Y,\widehat Z\big),
\]
hence $\mathbf B=0$.
\vspace{10pt}

(B5)$\Longrightarrow$(B3)\quad For any vector fields $X,Y$ on $M$, let
\[
(D_XY)^{\mathsf v}:=\left[X^{\mathrm h},Y^{\mathsf v}\right].
\]
Then the mapping
\[
D:\mathfrak X(M)\times\mathfrak X(M)\to\mathfrak X(M),\ (X,Y)\mapsto D_XY
\]
is well-defined. From the nice properties of the Lie bracket $\left[X^{\mathrm h},Y^{\mathsf v}\right]$ (see the Theorem in section 3 in Crampin's paper \cite{CrBianchi}, or verify it immediately) it follows that $D$ is a covariant derivative operator on $M$ with vanishing torsion.
\vspace{10pt}

$\left.\begin{array}{l}(\textrm{B3})\Longrightarrow\textrm{(B4)}\\\textrm{(B4)}\Longrightarrow\textrm{(B5)}
\end{array}\right\}$ These are obvious since $\vlift\nabla_{X\hl}\kal Y=[X\hl,Y\vl]$.

\vspace{10pt}

(B7)$\Longrightarrow$(B5)\quad We prove this by a simple coordinate calculation, using the local apparatus introduced in 4.3.

Let $X$ and $Y$ be vector fields on $M$,
\[
X\upharpoonright\mathcal U=X^i\frac\partial{\partial u^i},
\ Y\upharpoonright\mathcal U=Y^i\frac\partial{\partial u^i}.
\]
Then, over $\tau^{-1}(\mathcal U)$,
\[
X^{\mathrm h}=\left(X^j\frac\partial{\partial u^j}\right)^{\mathrm h}
=\left(X^j\circ\tau\right)\left(\frac\partial{\partial u^j}\right)^{\mathrm h}
=\left(X^j\circ\tau\right)
\left(\frac\partial{\partial x^j}-G^i_j\frac\partial{\partial y^i}\right),
\]
where the functions $G^i_j$ are the Christoffel symbols of the canonical connection of $(M,F)$. By its linearity,
\[
G^i_j=\left(\Gamma^i_{jl}\circ\tau\right)y^l,
\ \Gamma^i_{jl}\in C^\infty(\mathcal U),
\]
hence
\[
G^i_{jk}=\frac{\partial G^i_j}{\partial y^k}=\Gamma^i_{jk}\circ\tau.
\]
Thus
\begin{align*}
\big[X^{\mathrm h},Y^{\mathsf v}\big]
&=
\left(X^j\circ\tau\right)
\left[\frac\partial{\partial x^j}-G^i_j\frac\partial{\partial y^i},
Y^{\mathsf v}\right]\\
&=\left(X^j\circ\tau\right)
\left(\left[\frac\partial{\partial x^j},
\left(Y^i\circ\tau\right)\frac\partial{\partial y^i}\right]
+\left(Y^{\mathsf v}G^i_j\right)\frac\partial{\partial y^i}\right)\\
&=\left(X^j\circ\tau\right)
\left(\frac{\partial Y^i}{\partial u^j}\circ\tau
+\left(Y^k\circ\tau\right)G^i_{jk}\right)\frac\partial{\partial y^i}\\
&=\left(\left(X^j\frac{\partial Y^i}{\partial u^j}
+X^jY^k\Gamma^i_{jk}\right)\frac\partial{\partial u^i}\right)^{\mathsf v},
\end{align*}
which proves that $\left[X^{\mathrm h},Y^{\mathsf v}\right]$ is a vertical lift.
\vspace{10pt}

(B8)$\Longrightarrow$(B3)\quad As we saw in 3.2.(b), the Ehresmann connection $\mathcal H_D$ determined by $D$ is torsion-free and homogeneous. Lemmas 3, 4 guarantee that $\mathcal H_D$ is compatible with $F$, therefore, by the uniqueness of the canonical connection $\mathcal H$ of $(M,F)$, $\mathcal H_D=\mathcal H$. Thus for any vector fields $X,Y$ on $M$,
\[
X^{\mathrm h}=X^{\mathrm h_D}=X^{\mathsf c}-\overline{DX},
\]
and
\[
\left[X^{\mathrm h},Y^{\mathsf v}\right]=[X^{\mathsf c},Y^{\mathsf v}]+\left[Y^{\mathsf v},\overline{DX}\right]
=[X,Y]^{\mathsf v}+(D_YX)^{\mathsf v}=(D_XY)^{\mathsf v},
\]
which proves the implication.
\vspace{10pt}

(B3)$\Longrightarrow$(B8)\quad Note first that the covariant derivative operator $D$ in (B3) is indeed torsion-free, since the canonical connection $\mathcal H$ of $(M,F)$ is torsion-free.

As in the previous case, consider the Ehresmann connection $\mathcal H_D$. By our condition, and the torsion-freeness of $\mathcal H_D$, for any vector fields $X,Y$ on $M$ we have
\begin{align*}
\left[X^{\mathrm h},Y^{\mathsf v}\right]&=(D_XY)^{\mathsf v}=\left[X^{\mathsf v},\overline{DY}\right]
=[X^{\mathsf v},Y^{\mathsf c}]-\left[X^{\mathsf v},Y^{\mathrm h_D}\right]\\
&=[X,Y]^{\mathsf v}-\left[X^{\mathsf v},Y^{\mathrm h_D}\right]=\left[X^{\mathrm h_D},Y^{\mathsf v}\right].
\end{align*}
So, $X^{\mathrm h}-X^{\mathrm h_D}$ is a vertical vector field which commutes with each vertical lift, thus it is itself a vertical lift. Therefore, it is positive-homogeneous of degree 0 and degree 1 at the same time, which is possible only if $X^{\mathrm h}-X^{\mathrm h_D}=0$.
Thus $\mathcal H_D=\mathcal H$, and $\mathcal H_D$ is compatible with the Finsler function $F$. By Lemma 3, $\mathcal H_D$ generates the same parallelism as $D$, so, in view of Lemma 4, the parallel translations with respect to $D$ preserve the Finsler norms of the tangent vectors. This shows that the covariant derivative $D$ in (B3) satisfies the requirement of (B8).
\vspace{10pt}

(B3)$\Longrightarrow$(B9)\quad Due to the preceding argumentation, we have already known that $\mathcal H_D=\mathcal H$. Since, by Lemma 3 again, the $\mathcal H_D$-geodesics are the same parametrized curves as the $D$-geodesics, (B9) is indeed a consequence of (B3).
\vspace{10pt}

(B9)$\Longrightarrow$(B6)\quad Let $S$ be the canonical spray of $(M,F)$ and $\mathcal H_D$ the Ehresmann connection determined by the given covariant derivative. It can be checked immediately that the mapping
\[
S_D:TM\to TTM,\ v\mapsto S_D(v):=\mathcal H_D(v,v)
\]
is an affine spray. By our condition, it follows that for a regular curve\\ $\gamma:I\to M$,
\[
\ddot\gamma=S\circ\dot\gamma\mbox{ \emph{if and only if} }\ddot\gamma=S_D\circ\dot\gamma.
\]
Since any vector $v\in\mathring TM$ is the initial velocity of an
\[\textrm{\emph{$S$-geodesic = $S_D$-geo\-desic,}}\]
 we conclude that $S=S_D$, and hence $S$ is an affine spray.
\vspace{10pt}

(B6)$\Longrightarrow$(B7) Let $S$ be the canonical spray of $(M,F)$. If $S$ is an affine spray, then its spray coefficients $G^i:\tau^{-1}(\mathcal U)\to\mathbb R$ are of class $C^2$. Since these functions are positive-homogeneous of degree 2, it follows that fibrewise they are quadratic forms. So there exist smooth functions $\Gamma^i_{jk}$ on $\mathcal U$ such that
\[
G^i=\frac 12\left(\Gamma^i_{kl}\circ\tau\right)y^ky^l
\mbox{ \emph{and} }\Gamma^i_{kl}=\Gamma^i_{lk}
\quad(i,k,l\in\{1,\dots,n\}).
\]
Then the Christoffel symbols of the canonical connection $\mathcal H$ of $(M,F)$ are the smooth functions
\[
G^i_j:=\frac{\partial G^i}{\partial y^j}
=\left(\Gamma^i_{jk}\circ\tau\right)y^k,
\]
hence $\mathcal H$ is a linear connection.

This concludes the proof of the Proposition.
\end{proof}
If one, and hence each, of conditions (B1)--(B9) is satisfied, $(M,F)$ is said to be a \emph{Berwald manifold}. The covariant derivative $D$ appearing in (B3), (B4), (B8), (B9) is clearly unique; it will be called the \emph{base covariant derivative} of the Berwald manifold.
\setcounter{unit}0
\section{Aikou's characterization of Berwald manifolds}
\begin{unit}
Let an Ehresmann connection $\Eh:TM\times_MTM\rightarrow TTM$ be given, and let $\Ve$ be the vertical mapping associated to $\Eh$. For every vector field $\xi$ on $\splt M$, we define a kind of \emph{Lie derivative} operator on the tensor algebra of the $C^\infty(\splt M)$-module $\Sec(\splp)$, by prescribing its action
\begin{align*}
&\textrm{on \emph{functions} by }\lie_\xi f:=\xi f,\ f\in C^\infty(TM);\\
&\textrm{on \emph{sections} by }\lie_\xi \hll Y:=\Ve[\xi,\vlift \hll Y],\ \hll Y\in \Sec(\splp),
\end{align*}
and by extending it to the whole tensor algebra in such a way that $\lie_\xi$ satisfies the product rule of tensor derivations. Then, in particular, for any vector fields $X,Y$ on $M$ we get
\[\lie_{X\tl}\kal Y:=\Ve[X\tl,Y\vl]=\Ve[X,Y]\vl=\kal{[X,Y]}=\kal{\lie_XY},\]
so our Lie derivative is a natural extension of the `ordinary Lie derivative' on the base manifold.
\end{unit}

The Lie derivative of a section of $\splp$ with respect to the horizontal lift of a vector field on $M$ also has a nice dynamic interpretation.
\begin{lemma}
Let\/ $\Eh$ be an Ehresmann connection in $TM$, and let $X\hl$ be the $\Eh$-horizontal lift of $X\in\vm(M)$. Then for any section $\hll Y\in\Sec(\splp)$ and tangent vector $u\in \splt M$ we have
\begin{align*}
\big(\lie_{X\hl}\hll Y\big)(u)&=\lim_{t\to 0}\frac{(\varphi_{-t}\hl)_*\hll Y\big(\varphi_t\hl(u)\big)-\hll Y(u)}{t}\\
&=\big(t\longmapsto (\varphi_{-t}\hl)_*\hll Y(\varphi_t\hl(u))\big)'(0),
\end{align*}
where $\varphi:W\subset\mathbb R\times M\rightarrow M$ is the flow generated by $X$, $\varphi\hl$ is the extension of $\varphi$ described in Lemma \ref{lemflow}, and $\varphi_t\hl(u):=\varphi\hl(t,u)$ if $(t,u)$ is in the domain of $\varphi\hl$, and $t$ is fixed.
\end{lemma}
\begin{proof} (1) $\varphi_t\hl$ is a diffeomorphism between two open subsets of $\splt M$. However, for any vector $u$ in the domain of $\varphi_t\hl$, we may consider the derivative $((\varphi_t\hl)_*)_u$ as a mapping from $T_pM$ onto $T_{\varphi_t(p)}M$, where $p:=\tau(u)$, identifying $\varphi_t\hl$ with its restriction to $\splt_pM$, and identifying also the tangent spaces of a tangent space to $M$ with the tangent space itself. This interpretation of the derivative of $\varphi_t\hl$ will be applied automatically in what follows.

(2) Since, by Lemma \ref{lemflow}, $\varphi\hl$ is the flow generated by $X\hl$, our claim is an easy consequence of the dynamic interpretation of the ordinary Lie derivative:
\begin{align*}
\big(\lie_{X\hl}\hll Y\big)(u):&=\Ve[X\hl,\vlift\hll Y](u)=\\
&=\Ve\lim_{t\to 0}\frac{(\varphi_{-t}\hl)_*(\vlift\hll Y)(\varphi_t\hl(u))-(\vlift\hll Y)(u)}{t}\\
&=\lim_{t\to 0}\frac{(\varphi_{-t}\hl)_*\hll Y(\varphi_t\hl(u))-\hll Y(u)}{t},
\end{align*}
taking into account the linearity of $\vlift$ and the obvious relation \\ $\Ve\circ \vlift=1_{\Sec(\splp)}$.
\end{proof}
\begin{lemma}\label{lemb}
Hypothesis and notation as above. If \[b:\Sec(\splp)\times\Sec(\splp)\rightarrow C^\infty(\splt M)\]
is a type $(0,2)$ tensor, then
\[\lie_{X\hl}b=\lim_{t\to 0}\frac{(\varphi_t\hl)^*b-b}{t},\]
or, more precisely, for any vectors $u\in\splt M$; $v,w\in T_{\tau(u)}M$,
\[\tag{$\ast$} \big(\lie_{X\hl}b\big)_u(v,w)=\lim_{t\to 0}\frac{\big((\varphi_t\hl)^*b\big)_u(v,w)-b_u(v,w)}{t},\]
where $(\varphi_t\hl)^*$ denotes pull-back, given by
\[\big((\varphi_t\hl)^*b\big)_u(v,w):=b_{\varphi_t\hl(u)}\big(((\varphi_t\hl)_*)_u(v),((\varphi_t\hl)_*)_u(w)\big).\]
\end{lemma}
\begin{proof}
Let, for brevity, $p:=\tau(u)$. If $X(p)=0$, then both sides of $(\ast)$ vanish. Otherwise, there is a positive real number $\varepsilon$ and there are vector fields $Y$, $Z$ on $M$ such that
\[Y(\varphi_t(p))=((\varphi_t\hl)_*)_u(v)\textrm{ \emph{and} } Z(\varphi_t(p))=((\varphi_t\hl)_*)_u(w),\]
whenever $|t|<\varepsilon$. Hence, identifying the basic sections $\kal Y,\kal Z$ with their `principal parts' $Y\circ\tau,Z\circ\tau$, and applying the previous lemma, we obtain:
\begin{align*}
\big(\lie_{X\hl}b\big)_u(v,w)&=\big(\lie_{X\hl}b\big)(\kal Y,\kal Z)(u)=\\
&=\big(X\hl b(\kal Y,\kal Z)-b(\lie_{X\hl}\kal Y,\kal Z)-b(\kal Y,\lie_{X\hl}\kal Z)\big)(u)\\
&=\big(t\longmapsto b_{\varphi_t\hl(u)}\big(Y(\varphi_t(p)),Z(\varphi_t(p))\big)\big)'(0)-\\
&-\big(t\longmapsto b_u\big((\varphi_{-t}\hl)_*Y(\varphi_t(p)),w\big)+b_u\big(v,(\varphi_{-t}\hl)_*Z(\varphi_t(p))\big)\big)'(0)\\
&=\big(t\longmapsto b_{\varphi_t\hl(u)}\big(((\varphi_t\hl)_*)_u(v),((\varphi_t\hl)_*)_u(w)\big)-2b_u(v,w)\big)'(0)\\
&=\big(t\longmapsto((\varphi_t\hl)^*b)_u(v,w)\big)'(0)\\
&=\lim_{t\to0}\frac{((\varphi_t\hl)^*b)_u(v,w)-b_u(v,w)}{t}.
\end{align*}
\end{proof}
\begin{cor}
\emph{If $(M,F)$ is a Berwald manifold, $\Eh$ is its canonical connection, then the metric tensor of $(M,F)$ is constant along the flow generated by any $\Eh$-horizontal lift $Z\hl$ of $Z\in\vm(M)$.}
\end{cor}
\begin{proof}
By Lemma \ref{lemb}, it is enough to check that $\lie_{Z\hl}g=0$. Choosing two vector fields $X$, $Y$ on $M$, we calculate:
\begin{align*}
(\lie_{Z\hl}g)(\kal X,\kal Y)&=Z\hl g(\kal X,\kal Y)-g(\lie_{Z\hl}\kal X,\kal Y)-g(\kal X,\lie_{Z\hl}\kal Y)=\\
&=\frac{1}{2}Z\hl(X\vl Y\vl F^2)-g(\Ve[Z\hl,X\vl],\kal Y)-g(\kal X,\Ve[Z\hl,Y\vl])\mathop{=}^{\mathrm{(B3)}}\\
&=\frac{1}{2}Z\hl(X\vl Y\vl F^2)-g\big(\kal{D_ZX},\kal Y\big)-g\big(\kal X,\kal{D_ZY}\big)=\\
&=\frac{1}{2}\big(Z\hl(X\vl Y\vl F^2)-(D_ZX)\vl(Y\vl F^2)-X\vl((D_ZY)\vl F^2)\big)\mathop{=}^{\mathrm{(B3)}}\\
&=\frac{1}{2}\big(Z\hl(X\vl Y\vl F^2)-[Z\hl,X\vl](Y\vl F^2)-X\vl([Z\hl,Y\vl]F^2)\big)=\\
&=\frac{1}{2}X\vl Y\vl(Z\hl F^2)=0,
\end{align*}
since, by the compatibility of $\Eh$ and $F$, $Z\hl F^2=2FZ\hl F=0$.
\end{proof}
\begin{prop}[T.~Aikou \cite{Aikou}] \emph{A Finsler manifold
is a Berwald manifold if, and only if, there exists a Riemannian metric $g_M$ on $M$ such that for every vector field $Z$ on $M$, $\lie_{Z\hl}\kal g_M=0$. Here the horizontal lift is taken with respect to the canonical connection of $(M,F)$, and $\kal g_M$ is the natural lift of $g_M$ into $\mathsf T_2^0(\Sec(\splp))$ given on basic sections by
\[\kal g_M(\kal X,\kal Y):=(g_M(X,Y))\vl;\quad X,Y\in\vm(M).\]}
\end{prop}
\begin{proof}
\emph{Sufficiency.} Let $X$ and $Y$ be vector fields on $M$. Then, as above,
\begin{align*}
(\lie_{Z\hl}\kal g_M)(\kal X,\kal Y)&=Z\hl\kal g_M(\kal X,\kal Y)-\kal g_M\big(\Ve[Z\hl,X\vl],\kal Y\big)-\\
&-\kal g_M\big(\kal X,\Ve[Z\hl,Y\vl]\big)=Z\hl\big(g_M(X,Y)\big)\vl-\\
&-\kal g_M(\nabla_{Z\hl}\kal X,\kal Y)-\kal g_M\big(\kal X,\nabla_{Z\hl}\kal Y).
\end{align*}
Here, as it can be seen at once,
\[Z\hl\big(g_M(X,Y)\big)\vl=\big(Zg_M(X,Y)\big)\vl,\]
so by our condition $\lie_{Z\hl}\kal g_M=0$, it follows that
\[\kal g_M(\nabla_{z\hl}\kal X,\kal Y)+\kal g_M(\kal X,\nabla_{Z\hl}\kal Y)=\big(Zg_M(X,Y)\big)\vl.\]
Permuting $Z$, $X$ and $Y$ cyclically, we obtain
\begin{align*}
\kal g_M(\nabla_{X\hl}\kal Y,\kal Z)+\kal g_M(\kal Y,\nabla_{X\hl}\kal Z)&=\big(Xg_M(Y,Z)\big)\vl,\\
\kal g_M(\nabla_{Y\hl}\kal Z,\kal X)+\kal g_M(\kal Z,\nabla_{Y\hl}\kal X)&=\big(Yg_M(Z,X)\big)\vl.
\end{align*}
Adding both sides of the last two relations and subtracting the preceding one, we find
\begin{align*}\kal g_M\big(&\nabla_{X\hl}\kal Y+\nabla_{Y\hl}\kal X,\kal Z\big)+\kal g_M\big(\nabla_{Y\hl}\kal Z-\nabla_{Z\hl}\kal Y,\kal X\big)+\kal g_M\big(\nabla_{X\hl}\kal Z-\nabla_{Z\hl}\kal X,\kal Y\big)\\
&=\big(Xg_M(Y,Z)+Yg_M(Z,X)-Zg_M(X,Y)\big)\vl.
\end{align*}
By Lemma \ref{lemBnab}, the left-hand side of this relation can be written in the form
\[\kal g_M\big(2\nabla_{X\hl}\kal Y-\kal{[X,Y]},\kal Z\big)+\kal g_M\big(\kal{[Y,Z]},\kal X\big)+\kal g_M\big(\kal{[X,Z]},\kal Y\big),\]
so we obtain
\begin{align*}
2\kal g_M(\nabla_{X\hl}\kal Y,\kal Z)&=\big(Xg_M(Y,Z)+Yg_M(Z,X)-Zg_M(X,Y)\big)\vl+\\
&+\big(-g_M(X,[Y,Z])+g_M(Y,[Z,X])+g_M(Z,[X,Y])\big)\vl.
\end{align*}
If $D$ is the Levi-Civita derivative of $(M,g_M)$, then, by the Koszul formula, the right-hand side of the last relation is just
\[2\big(g_M(D_XY,Z)\big)\vl=2\kal g_M(\kal{D_XY},\kal Z),\]
hence
\[\kal g_M\big(\nabla_{X\hl}\kal Y-\kal{D_XY},\kal Z\big)=0\textrm{ \emph{for all} } X,Y,Z\in\vm(M).\]
This implies that
\[\vlift \nabla_{X\hl}\kal Y=(\nabla_XY)\vl\textrm{ \emph{for all} }X,Y\in\vm(M),\]
whence, by (B4), $(M,F)$ is a Berwald manifold.

\par\noindent\emph{Necessity.} Assume that $(M,F)$ is a positive definite Berwald manifold. By a celebrated observation of Z.~I.~Szab\'o \cite{Szabo}, there exists a Riemannian metric $g_M$ on $M$ whose Levi-Civita derivative is the base covariant derivative $D$ of $(M,F)$. (For an instructive, quite recent proof of this fact we refer to Vincze's paper \cite{Vincze}. In the next section
we shall see that it works under more general assumptions.)

$g_M$ satisfies our requirement: for any vector fields $X,Y,Z$ on $M$ we have
\begin{align*}
(\lie_{Z\hl}\kal g_M)&(\kal X,\kal Y)=Z\hl\kal g_M(\kal X,\kal Y)-\kal g_M(\lie_{Z\hl}\kal X,\kal Y)-\kal g_M(\kal X,\lie_{Z\hl}\kal Y)=\\
&=\big(Zg_M(X,Y)\big)\vl-\kal g_M(\Ve[Z\hl,X\vl],\kal Y)-\kal g_M(\kal X,\Ve[Z\hl,Y\vl])\mathop{=}^{\textrm{(B3)}}\\
&=\big(Zg_M(X,Y)-g_M(D_ZX,Y)-g_M(X,D_ZY)\big)\vl=\\
&=\big(Dg_M(Z,X,Y)\big)\vl=0,
\end{align*}
since $D$ is a metric derivative on $(M,g_M)$.
\end{proof}

\setcounter{section}6
\setcounter{lemma}{10}
\setcounter{unit}0
\section{On Matveev's generalization of Berwald manifolds}
\begin{unit}
In what follows, by a \emph{vector space} we shall mean a finite dimensional (but non-trivial) real vector space endowed with the canonical linear topology. Sometimes, tacitly, we also assume that the considered $n$-dimensional vector space $V$ is a manifold whose smooth structure is defined by a linear bijection $V\rightarrow \RR^n$.

We recall that in the context of these vector spaces we have a natural and efficient concept of differentiability of mappings. Namely, let $V$ and $W$ be vector spaces, and let $\lin(V,W)$ be the vector space of linear mappings of $V$ into $W$. Let $\uhz$ be an open subset of $V$. A mapping $\varphi:\uhz\rightarrow W$ is called \emph{differentiable at a point} $p\in\uhz$ if for some $\varphi'(p)\in\lin(V,W)$
\[\lim_{t\to 0}\frac{\varphi(p+tv)-\varphi(p)}{t}=\varphi'(p)(v),\quad v\in V.\]
$\varphi:\uhz\rightarrow W$ is \emph{differentiable} if it is differentiable at every point of $\uhz$; then its \emph{derivative} is the mapping
\[\varphi':\uhz\longrightarrow \lin(V,W),\quad p\longmapsto \varphi'(p).\]
$\varphi$ is \emph{twice differentiable} if $\varphi'$ is differentiable; the derivative of $\varphi'$ is a mapping
\[\varphi'':\uhz\longrightarrow\lin\big(V,\lin (V,W)\big),\quad p\longmapsto \varphi''(p),\]
called the \emph{second derivative} of $\varphi$. Here $\lin(V,\lin(V,W))$ may be canonically identified with the vector space $\lin^2(V,W)$ of bilinear mappings \mbox{$V\times V\rightarrow W$}. For further fine details we refer to \cite{Ker}, Ch.1.
\end{unit}
\begin{unit}
Let $V$ and $W$ be vector spaces, $r$ a real number, and $\uhz$ a (nonempty) subset of $V$. A mapping $\varphi:\uhz\rightarrow W$ is said to be \emph{positive-homogeneous of degree $r$}, briefly \emph{$r^+$-homogeneous}, if for each positive real number $\lambda$ and each $v\in\uhz$,
\[\lambda v\in\uhz\textrm{\emph{ and }} \varphi(\lambda v)=\lambda^r\varphi(v).\]

If, in particular, $\uhz$ is an open subset of $V$ with the property that $\lambda\uhz\subset\uhz$ for all positive $\lambda\in\RR$, and $f:\uhz\rightarrow\RR$ is a differentiable function, then, as it has been observed by Euler,
\[\textrm{\emph{$f$ is $r^+$-homogeneous if, and only if, }}f'(v)(v)=rf(v)\textrm{\emph{ for all }} v\in\uhz.\]
\end{unit}
\begin{unit}
We canonically identify the tangent space $T_pV$ of a vector space $V$ at a point $p$ with $V$. Then a (smooth) vector field on an open subset $\uhz$ of $V$ is just a smooth mapping $X:\uhz\rightarrow V$. As in the general theory of manifolds, we denote by $\vm(\uhz)$ the $C^\infty(\uhz)$-module of vector fields on $\uhz$. \emph{$\vm(\uhz)$ is generated by the constant vector fields} of the form
\[X:\uhz\longrightarrow V,\quad p\longmapsto X(p):=v\textrm{\emph{ for all }} p\in\uhz.\]
We denote by $Z_\uhz$ the \emph{radial vector field}
\[\uhz\longrightarrow V,\quad Z_\uhz(p):=p.\]
It plays the same role as the Liouville vector field (2.4) in the general theory.  For simplicity, the suffix $\uhz$ will be omitted.

If $f$ is a differentiable function on $\uhz$ and $X\in\vm(\uhz)$, then
\[(Xf)(p)=X(p)f=f'(p)(X(p)),\textrm{\emph{ for all }}p\in \uhz.\]
In particular,
\[(Zf)(p)=f'(p)(p),\ p\in \uhz.\]
It follows that \emph{$f$ is $r^+$-homogeneous if and only if ($\lambda\uhz\subset \uhz$ for all positive $\lambda\in\RR$ and) $Zf=rf$}.

Now let $f\in C^\infty(\uhz)$; $X,Y\in\vm(\uhz)$. Then at each point $p\in\uhz$,
\[X(Yf)(p)=f''(p)\big(X(p),Y(p)\big)+Y'(p)(X(p))(f)\]
therefore
\[[X,Y](p)=Y'(p)(X(p))-X'(p)(Y(p)).\]
From this we see that \emph{if $X$ is a constant vector field on $\uhz$ and $Z\in\vm(\uhz)$ is the radial vector field, then} $[X,Z]=X$.
\end{unit}
\begin{unit}
Let $V$ be a vector space, $k$ a positive integer, and let $\lin^k(V)$ denote the vector space of $k$-linear real-valued functions on $V$. If $\uhz$ is an open subset of $V$, then any $k$-linear function $A\in\lin^k(V)$ may be interpreted as a type $(0,k)$ \emph{tensor field} whose value $A_p$ at a point $p\in\uhz$ is just $A$, i.e.,
\[A_p(v_1,\dots,v_k):=A(v_1,\dots,v_k);\quad (v_1,\dots,v_k)\in V^k.\]
Equivalently, we may consider $A$ as a $C^\infty(\uhz)$-multilinear mapping \\$(\vm(\uhz))^k\rightarrow C^\infty(\uhz)$, given by
\[A(X_1,\dots,X_k)(p):=A(X_1(p),\dots,X_k(p));\quad X_1,\dots,X_k\in\vm(\uhz),\ p\in \uhz.\]
Keeping these in mind, for any vector field $X$ on $\uhz$ we may define the contracted tensor field $i_XA$ and the Lie derivative $\lie_XA$ in the usual manner. If $A$ is skew-symmetric, we may also speak of the exterior derivative $dA$.
\end{unit}
\begin{lemma}\label{1n}
Let\/ $\uhz$ be an open subset of a vector space $V$, and let $A\in\lin^k(V)$ be a $k$-form, considered as a type $(0,k)$ tensor field on $\uhz$. Then $\lie_ZA=kA$, where $Z\in\vm(\uhz)$ is the radial vector field.
\end{lemma}
\begin{proof}
It is enough to check the relation for constant vector fields $X_1,\dots,X_k$ in $\uhz$. Then
\begin{align*}
(\lie_ZA)(X_1,\dots,X_k):&=Z\big(A(X_1,\dots,X_k)\big)-\sum_{i=1}^kA(X_1,\dots,[Z,X_i],\dots,X_k)\\
&=\sum_{i=1}^kA(X_1,\dots,X_k)=kA(X_1,\dots,X_k)
\end{align*}
since the function $A(X_1,\dots,X_k)$ is constant, while $[Z,X_i]=-X_i$ \\$(i\in\{1,\dots,k\})$ as we have seen in 7.3.
\end{proof}
\begin{unit}
By a \emph{pre-Finsler norm} on a vector space $V$ we mean a function \mbox{$f:V\rightarrow\RR$} such that
\begin{enumerate}[(i)]
\item $f$ is of class $C^2$ on $V\nul$;
\item $f$ is $1^+$-homogeneous.
\end{enumerate}
Then $\psi:=\frac{1}{2}f^2$ is the \emph{energy} associated to $f$. A pre-Finsler norm $f:V\rightarrow\RR$ is said to be a \emph{gauge} if
\begin{enumerate}
\item[(iii)] $f(v)>0$ for all $v\in V\nul$;
\item[(iv)] $f$ is subadditive, i.e., $f(v+w)\leq f(v)+f(w)$ \emph{for all} $v,w\in V$.
\end{enumerate}

Notice that the energy $\psi$ is $2^+$-homogeneous and differentiable at $0$, with derivative $0\in V^*:=\lin(V,\RR)$. It follows also immediately that a gauge $f:V\rightarrow\RR$ is a convex function:
\[f\big((1-t)v+tw\big)\leq (1-t)f(v)+tf(w)\]
for all $v,w\in V$ and $t\in[0,1]$. Thus, by condition (i), the second derivatives
\[f''(u):V\times V\longrightarrow\RR,\quad u\in V\nul\]
are positive semidefinite.

A vector space equipped with a pre-Finsler norm or with a gauge will be called a \emph{pre-Finsler} or a \emph{gauge vector space}, respectively.
\end{unit}
\begin{unit}
Let $(V,f)$ be an $n$-dimensional gauge vector space, $n\geq2$. Then
\[B:=\big\{v\in V\big|f(v)\leq1\big\}\textrm{ \emph{and} } S:=\big\{v\in V\big|f(v)=1\big\}\]
are the \emph{f-unit ball} and the \emph{$f$-unit sphere} of $(V,f)$, respectively. The tangent space $T_aS$ of $S$ at a point $a\in S$ may be identified with the $(n-1)$-dimensional subspace $\kr(f'(a))$ of $V$. Notice that $a\notin T_aS$, since
\[f'(a)(a)=f(a)>0\]
by condition (ii) and (iii) in 7.5.

Let an orientation of $V$ be given, and let $\Omega:V^n\rightarrow\RR$ be the unique $n$-form such that $\int_B\Omega=1$. Then the $(n-1)$-form $\omega$ on $S$ given by
\[\omega_a(v_2,\dots,v_n):=\Omega(a,v_2,\dots,v_n),\quad a\in S;\ v_i\in T_aS\subset V,\ i\in\{2,\dots,n\}\]
orients $S$; this is the orientation of $S$ induced by the orientation of $V$ (cf. \cite{GHV} 3.21, Ex. 2). Equivalently, $\omega$ may simply be defined by the formula
\[\omega:=i_Z\Omega\upharpoonright S.\]
\end{unit}

The following observation is a slight generalization of Lemma 2 of Vincze's paper \cite{Vincze}, with essentially the same proof.
\begin{lemma}{\label{bsint}}
Let $(V,f)$ be a gauge vector space of dimension $n\geq2$. If \mbox{$h:V\rightarrow\RR$} is a $0^+$-homogeneous function, of class $C^1$ outside the zero, then
\[\int_B h=\frac 1n\int_Sh.\]
\end{lemma}
\begin{proof}
By a slight abuse of notation, we are going to use Stokes' formula. Since $Zh=i_Zdh=0$, and, by Lemma \ref{1n}, $\Omega=\frac 1n\lie_Z\Omega$, we obtain
\begin{align*}
\int_Bh:&=\int_Bh\Omega=\frac 1n\int_B\lie_Zh\Omega=\frac 1n\int_B(i_Z\circ d+d\circ i_Z)h\Omega\\
&=\frac 1n\int_Bd(i_Zh\Omega)=\frac 1n\int_Si_Zh\Omega=\frac 1n\int_S h\omega=:\frac 1n\int_Sh.
\end{align*}
\end{proof}
\begin{unit}
We shall now associate a Euclidean structure to a gauge, by the averaged metric construction of Matveev et al. \cite{MRTZ}.
\end{unit}
\begin{lemma}\label{gppos}
Let $(V,f)$ be a gauge vector space with energy function $\psi=\frac 12 f^2$. If for each $v,w\in V$,
\[b(v,w):=\int_S\big(u\longmapsto \psi''(u)(v,w)\big)\omega,\]
then $b$ is a positive definite scalar product on $V$.
\end{lemma}
\begin{proof}
Bilinearity and symmetry of $b$ are obvious, since for all $u\in V\nul$, the second derivative $\psi''(u):V\times V\rightarrow\RR$ has these properties. In addition, $\psi''(u)$ is positive semidefinite, since $\psi$ is also a convex function. If $v\in V\nul$ and $a:=\frac1{f(v)}v$, then $a\in S$ and
\begin{align*}
\psi''(a)(v,v)&=\psi''\left(\frac 1{f(v)}v\right)(v,v)=\psi''(v)(v,v)\\
&=\psi'(v)(v)=2\psi(v)=f(v)^2>0,
\end{align*}
taking into account that $\psi''$, $\psi'$ and $\psi$ are $0^+$-, $1^+$- and $2^+$-homogeneous, respectively, and applying repeatedly Euler's relation (7.2). Thus the function
\[u\in S\longmapsto \psi''(u)(v,v)\in\RR\]
with fixed $v\in V\nul$, is positive at the point $a\in S$, therefore, by continuity, it is positive also in a neighbourhood of $a$.

This proves the positive definiteness of $b$.
\end{proof}
\begin{unit}
We recall that a \emph{rough section} in a vector bundle $\pi:N\rightarrow M$ is any mapping $s:M\rightarrow N$ such that $\pi\circ s=1_M$. In what follows, we shall use the term `rough tensor field' in this sense.

Now suppose that $D$ is a covariant derivative operator on the manifold $M$. We say that a rough tensor field $A$ of type $(0,k)$ on $M$ is \emph{invariant by $D$-parallel translation}, if for any two points $p,q$ in $M$ and curve segment $\gamma$ from $p$ to $q$, for the parallel translation $P_\gamma:T_pM\rightarrow T_qM$ along $\gamma$ we have $P_\gamma^*A_q=A_p$, where
\[(P_\gamma^*A_q)(v_1,\dots,v_n):=A_q\big(P_\gamma(v_1),\dots,P_\gamma(v_n)\big);\ v_i\in T_pM,\ i\in \{1,\dots,n\}.\]
\end{unit}
\begin{lemma}\label{smooth}
Let $D$ be a covariant derivative on $M$. If a rough covariant tensor field on $M$ is invariant by $D$-parallel translation, then it is actually smooth.
\end{lemma}
\begin{proof}
The property in question is local, so it is enough to show the desired smoothness in a neighbourhood of an arbitrarily chosen point of $M$. So let $p\in M$, and let $\uhz$ be a normal neighbourhood of $p$. If $(e_i)_{i=1}^n$ is a basis of $T_pM$, then it can be extended to a frame $(E_i)_{i=1}^n$ for $TM$ over $\uhz$ by parallel translations along geodesics starting from $p$. Differential equation theory (smoothness of ODE solutions) guarantees that the vector fields $E_i$ are smooth. Now, if a rough covariant tensor field $A$ on $M$ is invariant by $D$-parallel translation, then the components of $A$ with respect to $(E_i)_{i=1}^n$ are constant; hence $A$ is smooth over $\uhz$.
\end{proof}
\begin{unit}
We say that a function $F:TM\rightarrow\RR$ is a \emph{pre-Finsler function}, resp. a \emph{gauge function} for $M$, if it is of class $C^2$ on $\splt M$ and $F_p:=F\upharpoonright T_pM$ is a pre-Finsler norm, resp. a gauge for each $p\in M$. $E:=\frac12F^2$ is the \emph{energy function} associated to the pre-Finsler function (or gauge function) $F$. A manifold equipped with a pre-Finsler function (resp. a gauge function) is said to be a \emph{pre-Finsler manifold} (resp. a \emph{gauge manifold}). A gauge manifold $(M,F)$ becomes a Finsler manifold, if $F$ is smooth on $\splt M$, and for each $p\in M$, $u\in T_pM\nul$
\[(F_p)''(u)(v,v)=0\textrm{ \emph{implies} }v\in \mathrm{span}(u);\]
see \cite{Ker}, Prop 4.5.
\end{unit}
\begin{unit}
Now we are in a position to introduce the main actor of this chapter, and to formulate and prove Matveev's generalization of Szab\'o's theorem on Riemann metrizability of the base covariant derivative of a Berwald manifold.

We say that a triplet $(M,F,D)$ is a \emph{Berwald\,--\,Matveev manifold} if $F$ is a gauge function for $M$ and $D$ is a torsion-free covariant derivative on $M$ which is \emph{compatible with $F$} in the sense that the parallel translations with respect to $D$ preserve the $F$-norms of tangent vectors to $M$. The next Proposition assures that this compatibility condition determines the covariant derivative $D$ uniquely.
\end{unit}
\begin{prop}\label{gm}
\emph{Let $(M,F,D)$ be a Berwald\,--\,Matveev manifold of dimension $n$, $n\geq2$. For every point $p\in M$, let
\renewcommand*\labelitemi{}
\begin{itemize}
\item $F_p:=F\upharpoonright T_pM$; $E_p:=E\upharpoonright T_pM$, $E=\frac12F^2$;
\item $B_p\subset T_pM$ the unit $F_p$-ball;
\item $S_p\subset T_pM$ the unit $F_p$-sphere;
\item $Z_p:T_pM\rightarrow T_pM$, $v\mapsto Z_p(v):=v$ the radial vector field on $T_pM$;
\item $\Omega_p$ the unique volume form on $T_pM$ such that $\int_{B_p}\Omega=1$;
\item $\omega_p:=i_{Z_p}\Omega_p\upharpoonright S_p$ the induced volume form on $S_p$.
\end{itemize}}

\emph{Define a type $(0,2)$ rough tensor field $g_M$ on $M$ by prescribing its value $(g_M)_p$ at a point $p\in M$ according to Lemma \ref{gppos}, that is, by the rule
\[(g_M)_p(v,w):=\int_{S_p}\big(u\longmapsto (E_p)''(u)(v,w)\big)\omega_p\ ;\quad v,w\in T_pM.\]
Then $g_M$ is a (positive definite) Riemannian metric on $M$ whose Levi-Civita derivative is $D$.}
\end{prop}
\begin{proof}
By Lemma \ref{gppos}, $(g_M)_p$ is a positive definite scalar product on $T_pM$ for each $p\in M$. So it is enough to check that $g_M$ is invariant by $D$-parallel translation: then Lemma \ref{smooth} implies that $g_M$ is a Riemannian metric on $M$, and it follows at once that the Levi-Civita derivative of $(M,g_M)$ is just $D$.
Let $p,q\in M$, and let $\gamma:[0,1]\rightarrow M$ be a curve segment connecting $p$ with $q$, i.e. $\gamma(0)=p$, $\gamma(1)=q$. Consider the parallel translation $P_\gamma:T_pM\rightarrow T_qM$ along $\gamma$.
\begin{claim}
$P_\gamma^*\Omega_q=\pm\Omega_p$.
\end{claim}

To see this, let us first note that $P_\gamma$ preserves the $F$-norms of the tangent vectors to $M$ by our compatibility condition, so it is a diffeomorphism of $B_p$ onto $B_q$. Thus, applying the `change of variables' formula for the integral of differential forms, we obtain
\[\int_{B_p}P_\gamma^*\Omega_q=\pm\int_{B_q}\Omega_q=\pm 1.\]
This implies (by the uniqueness of $\Omega_p$) the desired relation.

\begin{claim}
For fixed $v,w\in T_pM$, let
\[\varphi(z):=(E_q)''(z)(P_\gamma(v),P_\gamma(w)),\quad z\in T_q M\setminus\{0\}.\]
Then $\varphi$ is $0^+$-homogeneous and
\[\varphi\circ P_\gamma(z)=(E_p)''(z)(v,w),\quad z\in T_p M\setminus\{0\}.\]
\end{claim}

The first assertion is obvious, since $(E_q)''$ is $0^+$-homogeneous. The formula for $\varphi\circ P_\gamma$ may be checked immediately, using the definition of the second derivative of $E_q$ at $P_\gamma(z)$ and taking into account that $P_\gamma$ is a linear mapping.
\begin{claim}
$P_\gamma^*(g_M)_q=(g_M)_p$.
\end{claim}
Let $v,w\in T_pM$. Then
\begin{align*}
(P_\gamma^*(g_M)_q)(v,w)&=(g_M)_q(P_\gamma(v),P_\gamma(w))\\
&:=\int_{S_q}\big(z\longmapsto(E_q)''(z)(P_\gamma(v),P_\gamma(w))\big)\omega_q\\
&\overset{\textrm{Lemma \ref{bsint}}}=\frac1n\int_{B_q}\big(z\longmapsto(E_q)''(z)(P_\gamma(v),P_\gamma(w))\big)\Omega_q=\\
&=\frac1n\int_{B_q}\varphi\,\Omega_q\\
&\overset{\textrm{Claim 1}}=\pm\frac1n\int_{B_q}\varphi(P_\gamma^{-1})^*\Omega_p\\
&\overset{\textrm{change of variables}}=\frac1n\int_{B_p}(\varphi\circ P_\gamma)\Omega_p\\
&\overset{\textrm{Lemma \ref{bsint}}}=\int_{S_p}(\varphi\circ P_\gamma)\omega_p\\
&\overset{\textrm{Claim 2}}=\int_{S_p}\big(u\longmapsto (E_p)''(u)(v,w)\big)\omega_p\\
&=(g_M)_p(v,w).
\end{align*}

This concludes the proof.
\end{proof}

\begin{unit}
Let $(M,F)$ be a \emph{pre-Finsler manifold}; $[a,b]\subset\RR$, where $a\neq b$, a compact interval, and choose two points, $p$ and $q$, in $M$. Denote by $\mathcal C(p,q)$ the set of all $C^1$ curve segments $\gamma:[a,b]\to M$ such that $\gamma(a)=p$, $\gamma(b)=q$. Define the \emph{energy functional}
\[\mathbb E:\mathcal C(p,q)\to\mathbb R\]
by
\[\mathbb E(\gamma):=\int_a^bE\circ\dot\gamma=\int_a^bE(\dot\gamma(t))dt,
\quad\gamma\in\mathcal C(p,q).\]
The regular extremals of $\mathbb E$ (i.e., the critical `points' $\gamma\in\mathcal C(p,q)$ with $\dot\gamma(t)\neq 0$ for all $t\in[a,b]$) are said to be the \emph{geodesics} corresponding to (or of) $F$. In terms of coordinates, a regular curve $\gamma$ in $\mathcal C(p,q)$ is a geodesic of $F$ if, and only if, in any induced chart $\big(\tau^{-1}(\mathcal U);(x^i,y^i)_{i=1}^n\big)$ such that $\mathrm{Im}(\gamma)\cap\mathcal U\neq\emptyset$, the \emph{Euler\,--\,Lagrange equations}
\[\frac{\partial E}{\partial x^i}\circ\dot\gamma
-\left(\frac{\partial E}{\partial y^i}\circ\dot\gamma\right)'=0,
\quad i\in\{1,\dots,n\}\]
are satisfied. If, in particular, $F$ is a Finsler function, then the concept of an $F$-geodesic just introduced yields the same curves as our earlier definition in 4.2(a).
\end{unit}
\begin{remark}
One may also consider the arclength-functional
\[
\mathbb F:\mathcal C(p,q)\to\mathbb R,\quad\gamma\mapsto
\mathbb F(\gamma):=\int_a^bF\circ\dot\gamma=\int_a^bF(\dot\gamma(t))dt.
\]
It is not difficult to show (see, e.g., \cite{GH}, p.\ 185) that the set of the geodesics corresponding to $F$ coincides with the set of positive constant speed extremals of $\mathbb F$.
\end{remark}

\begin{prop}
\emph{Let $(M,F,D)$ be an at least two-dimensional Berwald\,--\,Matveev manifold with associated Riemannian metric $g_M$ defined by Proposition \ref{gm}. Then any geodesic of the Riemannian manifold $(M,g_M)$ is also a geodesic of $F$.}
\end{prop}
\begin{proof}
Let $E_R$ be the energy function associated to the Riemannian metric $g_M$, given by
\[E_R(v):=\frac12(g_M)_{\tau(v)}(v,v),\quad v\in TM.\]
We define the function
\[
\widetilde E:=E_R+E,
\]
and let $\widetilde F:=\sqrt{2\widetilde E}$. Then $\widetilde F$ is of class $C^2$ on $\mathring TM$ and satisfies (F2). Since at each point $p\in M$ and for any tangent vectors $u\in\mathring T_pM$, $v,w\in T_pM$ we have
\[
\big(\widetilde E_p\big)''(u)(v,w):=(g_M)_{\tau(u)}(v,w)+(E_p)''(u)(v,w),
\]
and on the right-hand side of this relation $(E_p)''(u)$ is positive semidefinite, while $(g_M)_{\tau(u)}$ is positive definite, it follows that
\[\big(\widetilde E_p\big)''(u)
=\big(\widetilde E\upharpoonright T_pM\big)''(u)\]
is positive definite. This implies that $\widetilde F$ satisfies condition (F4) in 4.1, therefore $\big(M,\widetilde F\big)$ is a Finsler manifold of class $C^2$. Since the parallel translations with respect to $D$ preserve the $E$-norms and, by Proposition \ref{gm}, also the $E_R$-norms of the tangent vectors to $M$, it follows that they also preserve the $\widetilde E$-norms, and hence the $\widetilde F$-norms. So, by Proposition 6, $\big(M,\widetilde F\big)$ is a Berwald manifold with Finsler function of class $C^2$, and the set of geodesics of $\widetilde F$ coincides with the set of geodesics corresponding to $F_R:=\sqrt{E_R}$.

In an induced chart $\big(\tau^{-1}(\mathcal U),(x^i,y^i)_{i=1}^n\big)$, the Euler\,--\,Lagrange equations of the energy functional
\[
\widetilde{\mathbb E}:\mathcal C(p,q)\to\mathbb R,\quad\gamma\mapsto
\widetilde{\mathbb E}(\gamma):=\int_a^b\widetilde E\circ\dot\gamma
=\int_a^bE_R\circ\dot\gamma+\int_a^bE\circ\dot\gamma
\]
take the form
\begin{align*}
0&=\frac{\partial E}{\partial x^i}\circ\dot\gamma
-\left(\frac{\partial E}{\partial y^i}\circ\dot\gamma\right)'
+\left(\frac{\partial E_R}{\partial x^i}\circ\dot\gamma
-\left(\frac{\partial E_R}{\partial y^i}\circ\dot\gamma\right)'\right)\\
&=\frac{\partial E}{\partial x^i}\circ\dot\gamma
-\left(\frac{\partial E}{\partial y^i}\circ\dot\gamma\right)',
\end{align*}
since, as we have just seen, if $\gamma$ is a geodesic of $\widetilde F$, then it is also a geodesic of $F_R$ -- and vice versa. Thus it follows that the geodesics of $F_R$, i.e., of the Riemannian manifold $(M,g_M)$, are also geodesics of the gauge function $F$.
\end{proof}

\begin{unit}
We present a further natural and simple construction to show that if $(M,F,D)$ is a Berwald\,--\,Matveev manifold, then $D$ is the Levi-Civita derivative of a Riemannian metric on $M$.

In what follows, by an \emph{ellipsoid} on a vector space $V$ we mean the unit ball of a positive definite scalar product $b:V\times V\rightarrow \RR$, i.e., a set of the form $\E(b):=\{v\in V|b(v,v)\leq 1\}$. Ellipsoids are preserved by linear isomorphisms, namely, if $\Phi:V\rightarrow W$ is an isomorphism, then
\[\Phi\big(\E(b)\big)=\E\big((\Phi^{-1})^*b\big).\]
By the classical Loewner\,--\,Behrend theorem (see e.g.\ \cite{Berg}), if $V$ is endowed with a Lebesgue measure and $K$ is a compact subset of $V$ with non-empty interior, then there exists a unique least-volume ellipsoid containing $K$. We call this ellipsoid the \emph{Loewner ellipsoid} determined by $K$.
\end{unit}

Now, keeping the notation of Proposition \ref{gm}, consider a Berwald\,--\,Matveev manifold $(M,F,D)$. At each point $p\in M$, denote by $\E(b_p)$ the Loewner ellipsoid determined by the unit $F_p$-ball $B_p$. Then
\[g_L:p\in M\longmapsto (g_L)_p:=b_p\in\lin^2(T_pM,\RR)\]
is a rough Riemannian metric on $M$; we show that it is actually a Riemannian metric.

As in the proof of Proposition \ref{gm}, let $p$ and $q$ be two points in $M$, and let $\gamma$ be a curve segment connecting $p$ with $q$. Then, by our previous remark,
\[P_\gamma\big(\E(b_p)\big)=\E\big((P_\gamma^{-1})^*b_p\big)\]
is an ellipsoid in $T_qM$. Since $P_\gamma$ preserves the $F$-norm of tangent vectors to $M$, we have $B_q\subset P_\gamma\big(\E(b_p)\big)$. If $\E_q$ is another ellipsoid containing $B_q$, then $P_\gamma^{-1}(\E_q)$ is an ellipsoid in $T_pM$ containing $B_p$, and
\begin{align*}
\int_{\E_q}\Omega_q&\overset{\textrm{Claim 1, 7.10}}=\pm\int_{\E_q}(P_\gamma^{-1})^*\Omega_p\\
&\overset{\textrm{change of variables}}=\int_{P_\gamma^{-1}(\E_q)}\Omega_p\geq \int_{\E(b_p)}\Omega_p\\
&\overset{\textrm{Claim 1, 7.10}}=\pm\int_{\E(b_p)}P_\gamma^*\Omega_q\overset{\textrm{change of variables}}=\int_{P_\gamma(\E(b_p))}\Omega_q\\
&=\int_{\E((P_\gamma^{-1})^*b_p)}\Omega_q.
\end{align*}
This implies that $P_\gamma\big(\E(b_p)\big)$ is the least-volume ellipsoid containing $B_q$, there\-fore $\E\big((P_\gamma^{-1})^*b_p\big)=\E(b_q)$ and hence $P_\gamma^*(g_L)_q=(g_L)_p$. Thus, by Lemma \ref{smooth}, $g_L$ is smooth, so it is indeed a Riemannian metric on $M$. It is clear from the construction that the Levi-Civita derivative for $g_L$ is the given covariant derivative $D$.

We note finally that \emph{if the holonomy group of $D$ is irreducible at a point of $M$}, then $g_L$ is proportional to the Riemannian metric $g_M$ constructed above. In general, if $g_1$ and $g_2$ are two Riemannian metrics on $M$ such that $Dg_1=Dg_2=0$, then -- under the irreducibility of the holonomy group of $D$ -- we have $g_2=\lambda g_1$, where $\lambda$ is a positive real number.


\begin{thebibliography}{99}
\bibitem{Aikou} T.~Aikou, Some remarks on Berwald manifolds and Landsberg manifolds, Acta Math.\ Acad.\ Paed.\ Ny\'iregyh\'aziensis \textbf{26} (2010), 139--148.
\bibitem{Barthel}
W.~Barthel, Nichtlineare Zusammenh\"ange und deren Holonomiegruppen, J.~Reine Angew.\ Math.\ \textbf{212} (1963), 120--149.
\bibitem{Berg} M.~Berger, Geometry I, Springer Verlag, Berlin and Heidelberg, 1987.

\bibitem{Ber}
L.~Berwald, Untersuchung der Kr\"ummung allgemeiner metrischer R\"aume auf Grund des in ihnen herrschenden Parallelismus, Math.~Z.\ \textbf{25} (1926), 40--73; \textbf{26} (1927), 176.
\bibitem{CrBianchi}
M.~Crampin, Generalized Bianchi identities for horizontal distributions, Math.\ Proc.\ Camb.\ Phil.\ Soc.\ \textbf{94} (1983), 125--132.
\bibitem{GH}
M.~Giaquinta and S.~Hildebrandt, Calculus of Variations II, Springer Verlag, Berlin and Heidelberg, 2004.
\bibitem{GHV}
W.~Greub, S.~Halperin and R.~Vanstone, Connections, Curvature, and Cohomology, Vol.~I, Academic Press, New York and London, 1972.
\bibitem{Ker}
D.~Cs.~Kert\'esz, On the geometry of Finsler vector spaces, MSc Thesis, Debrecen, 2011.
\bibitem{L}
R.~L.~Lovas, A note on Finsler\,--\,Minkowski norms, \textit{Houston J.\ Math.} \textbf{33} (2007), 701--707.
\bibitem{M1}
V.~S.~Matveev, Riemannian metrics having common geodesics with Berwald metrics,
Publ.\ Math.\ Debrecen \textbf{74} (3-4) (2009), 454--470.
\bibitem{MRTZ}
V.~S.~Matveev, H.-B.~Rademacher, M.~Troyanov and A.~Zeghib, Finsler conformal
Lichnerowicz\,--\,Obata conjecture, Annales de l'Institut Fourier \textbf{59}
(3) (2009), 937--949.
\bibitem{M2}
V.~S.~Matveev and M.~Troyanov, The Binet\,--\,Legendre ellipsoid in Finsler
geometry, \textsf{arXiv:1104.1647}.
\bibitem{Szabo} Z.~I.~Szab\'o, Positive definite Berwald spaces (Structure theorems), Tensor N.S. \textbf{35} (1981) 25--39.
\bibitem{LSz-Glob}
J.~Szilasi and R.~L.~Lovas, Some aspects of differential theories, in: Handbook of
Global Analysis, Elsevier (2007), 1071--1116.
\bibitem{SzV}
J.~Szilasi and Cs.~Vincze, A new look at Finsler connections and special Finsler manifolds, Acta Math.\ Acad.\ Paed.\ Ny\'i regyh\'aziensis \textbf{16} (2000), 33--63.
\bibitem{SzZ-ArXiV}
Z.~Szilasi, On the projective theory of sprays with applications to Finsler geometry, PhD Thesis, Debrecen 2010, \textsf{arXiv:0908.4384}.
\bibitem{Vincze} Cs.~Vincze, A new proof of Szab\'o's theorem on the Riemann metrizability of Berwald manifolds, Acta Math.\ Acad.\ Paed Ny\'iregyh\'aziensis  \textbf{21} (2005), 199--204.




\end{thebibliography}
\end{document}